\documentclass[12pt,a4paper]{amsart}
\usepackage[margin=.9in]{geometry}
\usepackage{microtype}
\usepackage[cal=euler]{mathalfa}
\usepackage{tikz,tikz-cd}

\numberwithin{equation}{section}
\input{commands_new}
\def\bts{\bigotimes}
\oper{Mor}

\def\Om{\Omega}
\def\bb{\mathbf}
\opr\fSet{Set_f}
\opr\Gm{G_m}

\forcsvlist\oper{Sym,Map,Ab,Cor,crit,Var,Bq,sgn,Exp,el}
\dmat\De{\Delta}
\dmat\bDe{\bb\De}
\def\ze{z}
\oper{supp}
\def\oh{\frac12}

\oper{rk}

\dmat\Te{\Theta}
\oper{limsup}
\dmat\msp{\Theta} 
\def\kp{a}
\def\eg{e.g.\ }
\def\cc{\mathbf c}
\oper{HN}
\def\Z{Z^\cc}

\oper{Ch}
\oper{ch}
\def\cexp{\mathbf{exp}}
\def\clog{\mathbf{log}}
\oper{Col}
\def\prl{\parallel}
\def\bOm{\bar A}
\def\hOm{A}
\oper{Col}
\def\Cola{\Col^1}
\def\bmm{\mathbf m}
\dmat\Z{\mathsf Z}
\oper{Arg}
\def\oprod{\prod^\curvearrowright}
\def\io{\iota}

\tikzset{
crc1/.style={circle,draw,inner sep=3pt},
crc2/.style={circle,draw=black,fill,inner sep=2pt},
}
\def\bin{\mathrm{bin}}

\usepackage[colorlinks,allcolors=blue]{hyperref}
\def\mytitle{

\title{Operadic approach to wall-crossing}
\author{Sergey Mozgovoy}
\address{School of Mathematics, Trinity College Dublin, Dublin 2, Ireland
\newline\indent
Hamilton Mathematics Institute, Trinity College Dublin, Dublin 2, Ireland}
\email{mozgovoy@maths.tcd.ie}

\begin{abstract}
Wall-crossing phenomena are ubiquitous in many problems of algebraic geometry and theoretical physics.
Various ways to encode the relevant information and the need to track the changes under the variation of parameters lead to rather complicated transformation rules and non-trivial combinatorial problems.
In this paper we propose a framework, reminiscent of collections and plethysms in the theory of operads,
that conceptualizes those transformation rules.
As an application we obtain new streamlined proofs of some existing wall-crossing formulas as well as some new formulas related to attractor invariants.
\end{abstract}

\maketitle

}

\begin{document}
\mytitle

\section{Introduction}
In the study of wall-crossing phenomena in algebraic geometry and theoretical physics
we often encounter the following situation.
Assume that we have a graded associative algebra $\cR=\bop_{\ga\in \bN^r}\cR_\ga$
over $\bQ$ (not necessarily commutative)
and two families of elements (for example, corresponding to counting of objects in some category \wrt different stability parameters) 
\begin{equation}
A=(A(\ga)\in \cR_\ga)_{\ga\in S},\qquad
B=(B(\ga)\in \cR_\ga)_{\ga\in S},\qquad S=\bN^r\ms\set0,
\end{equation}
related by equations
\begin{equation}\label{relAB}
B(\ga)=\sum_{\al_1+\dots+\al_n=\ga}
F(\al_1,\dots,\al_n)\cdot A(\al_1)\dots A(\al_n),\qquad \ga\in S,
\end{equation}
for some collection of elements $F(\al)\in\bQ$, $\al\in S^*=\bigsqcup_{n\ge1}S^n$. We will write $B=F*A$ in this case.
Assume now that we have another collection of elements $F':S^*\to\bQ$ and let $C=F'*B$.
As one would expect, there is an operation on collections $(F',F)\mto F'\circ F$, called plethysm, such that 
$(F'\circ F)* A=F'*(F*A)=C$.
The idea to study such collections goes back to \cite{joyce_configurations}.
It turns out to be rather fruitful as we can forget about the nature of the algebra $\cR$ and families $A,\,B,\,C$ and concentrate entirely on the set of collections $\Col_\bQ=\set{F:S^*\to\bQ}$ equipped with plethysm and other operations. 

\medskip
A reader familiar with operads (see \eg \cite{getzler_operads,loday_algebraic}) 
will immediately recognize the above picture.
One starts with an appropriate symmetric monoidal category \cV, for example the category of vector spaces.
Then one defines a symmetric collection 
(or $\bS$-object, or species \cite{joyal_une})
to be a functor $F:\bS\to\cV$, where $\bS$ is the groupoid of finite sets.
Such collection defines an endofunctor, called a Schur functor,
\begin{equation}
\cV\to\cV,\qquad
V\mto F*V=\bop_{n\ge1}F(n)\ts_{S_n}V^{\ts n}.
\end{equation}
Similarly to the above situation, the category of symmetric collections $\Col(\cV)=\Fun(\bS,\cV)$ can be equipped with a plethysm $\circ$ (a tensor product)
such that $(F'\circ F)*V=F'*(F*V)$ for $F',\,F\in\Col(\cV)$ and $V\in\cV$.
An operad is a monoid object in $\Col(\cV)$ with respect to plethysm.
Having mentioned the similarities between the two notions, we should also discuss their differences.
First of all, the target of our collections is the ring $\bQ$ instead of a symmetric monoidal category. 
One should think about it as the Grothendieck group (with rational coefficients) of the symmetric monoidal category of finite-dimensional vector spaces.
Secondly, the domain of our collections is quite different from the domain of symmetric collections.
However, recently there was developed a framework of operadic categories to deal with generalized operads \cite{batanin_operadic,lack_operadic}.
We will see that our domain $\bigsqcup_{n\ge1}S^n$ can be interpreted as such an operadic category.

\medskip
While the relation of our collections to collections in the theory of operads is not strictly necessary, it provides a useful insight and techniques that can be used in the context of wall-crossing.
In particular, the free operad construction, using the language of trees, can be transferred to our setting.
It provides a universal construction of collections inverse \wrt plethysm.
Note that if \eqref{relAB} is satisfied and $F(\al)=1$ for all $\al\in S$, then we can express $A$ in terms of $B$ recursively, but
it is useful to be able to express $A=F\inv*B$ for an explicitly defined collection $F\inv$.

\medskip
Let us go back to wall-crossing.
Assume that for any central charge \Z,
meaning a linear map $\Z:\Ga=\bZ^r\to\bC$
with $\Z(S)$ contained in the upper half-plane,
we have elements $A_{\Z}(\ga)\in \cR_\ga$,
for $\ga\in S$.
For any ray $\ell\sbs\bC$ in the upper half-plane, consider the series
\begin{equation}
A_{\Z,\ell}=1+\sum_{\Z(\ga)\in\ell} A_{\Z}(\ga)\in\what \cR
\end{equation}
in the completion of the algebra $\cR$.
The basic wall-crossing formula is the statement
that the series
\begin{equation}\label{bwc}
A=1+\sum_{\ga\in S} A(\ga)
=\oprod_\ell A_{\Z,\ell}
\end{equation}
is independent of $\Z$, 
where the product runs over rays $\ell$ in the upper half-plane ordered clockwise.
This statement depends on the nature of invariants $A_\Z(\ga)$ and the algebra $\cR$ (it can be a motivic Hall algebra, a quantum affine plane or some other algebra depending on the context and the goals of a study),
but in all contexts this statement is proved using the uniqueness of Harder-Narasimhan filtrations \cite{reineke_harder-narasimhan,joyce_configurations,kontsevich_stability}.
For us, however, the above wall-crossing formula will be just an assumption and a starting point of our approach.
Let us write $\ga>_\Z\ga'$ if $\Arg\Z(\ga)>\Arg\Z(\ga')$,
for $\ga,\ga'\in S$.
Then we can write \eqref{bwc} in the form
\begin{equation}\label{bwc2}
A(\ga)=\sum_{\ov{\al_1+\,\dots\,+\al_n=\ga}
{\al_1>_\Z\, \dots\, >_\Z\al_n}}
A_\Z(\al_1)\dots A_\Z(\al_n),\qquad \ga\in S.
\end{equation}
The last formula can be written using collections as
\begin{equation}
A=s_\Z*A_\Z,
\end{equation}
where the collection $s_\Z:S^*\to\bQ$, called the \idef{Harder-Narasimhan (HN) collection}, is given by
\begin{equation}
s_\Z(\al)
=\begin{cases}
1&\al_1>_\Z\dots>_\Z\al_n\\
0&\text{otherwise}
\end{cases}\qquad\qquad\al\in S^n.
\end{equation}
One can see from \eqref{bwc2} that elements $A_\Z(\ga)$ can be recursively expressed in terms of elements $A(\ga)$.
Therefore, assuming that the wall-crossing formula is satisfied, the elements $A_\Z(\ga)$ for all central charges $\Z$ are uniquely determined by the family $A$.
Collection $s_\Z$ is invertible with respect to the plethysm and, for any other central charge $\Z'$, we have
\begin{equation}
A_\Z=s_\Z\inv* A,\qquad 
A_{\Z'}=(s_{\Z'}\inv\circ s_\Z)* A_\Z,
\end{equation}
where the collections $s_\Z\inv$ and $s_{\Z'}\inv\circ s_\Z$
have explicit expressions (see \S\ref{sec:HN} and \cite{reineke_harder-narasimhan,joyce_configurations}).

\medskip
Now assume that $\Ga=\bZ^r$ is equipped with a skew-symmetric form $\ang{-,-}$ such that
\begin{equation}
\ang{\ga,\ga'}=0\imp [\cR_\ga,\cR_{\ga'}]=0,
\end{equation}
where $[-,-]$ is the induced Lie bracket on $\cR$.
Let $\Z:\Ga\to\bC$ be a generic central charge, meaning that $\Z(\ga),Z(\ga')\in\ell$ implies that $\ga,\ga'$ are proportional, hence $[\cR_\ga,\cR_{\ga'}]=0$.
Then we define a new family of invariants $\bOm_\Z(\ga)$ (called rational DT invariants in an appropriate context) by the formula
\begin{equation}
\bOm_{\Z,\ell}
=\sum_{\Z(\ga)\in\ell}\bOm_{\Z}(\ga)
=\log\rbr{A_{\Z,\ell}}.
\end{equation}
In \S\ref{sec:exp-log} we will define a collection $\clog_\prl:S^*\to\bQ$ (as well as its inverse $\cexp_\prl:S^*\to\bQ$ \wrt the plethysm) such that the family of invariants $\bOm_\Z$ can be expressed as
\begin{equation}\label{bom1}
\bOm_\Z=\clog_\prl* A_\Z=(\clog_\prl\circ s_\Z\inv)*A.
\end{equation}
While it is not an emphasis of this paper, one should mention that there is a parallel Lie-theoretic point of view which is better suited for specializations and is adopted in the theory of wall-crossing structures \cite{kontsevich_wall}.
The above equation implies that,
for two generic stability parameters $\Z,\,\Z'$, we have
\begin{equation}\label{bom2}
\bOm_{\Z'}=(\clog_\prl\circ s_{\Z'}\inv\circ s_\Z\circ\cexp_\prl)* \bOm_\Z,
\end{equation}
which a priori allows us to express invariants $\bOm_{\Z'}(\ga)$ in terms of invariants $\bOm_{\Z}(\ga)$ using multiplication in $\cR$.
However, it follows from the results of \cite{joyce_configurations} that only the induced Lie bracket on $\cR$ is required (see \S\ref{app1}).


\medskip
Among all stability parameters, there is a special one, called attractor stability, which is suggested by physics \cite{denef_supergravity,denef_split,cheng_dying,manschot_walla,alexandrov_attractor,alexandrov_s}
and leads to particularly simple invariants in many applications \cite{beaujard_vafa,mozgovoy_attractor}.
We define attractor invariants (also called \qq{immortal} BPS indices \cite{cheng_dying,manschot_walla}
or initial data \cite{kontsevich_wall,gross_canonicala}) as
\cite{mozgovoy_attractor}
\begin{equation}
\bOm_*(\ga)=\bOm_{\te_\ga}(\ga),
\end{equation}
where $\te_\ga$, called an attractor stability,
is a generic perturbation of the self-stability $\ang{-,\ga}$.
We will prove in Theorem \ref{th:bom*1} that $\bOm_*(\ga)=\bOm_\te(\ga)$ for $\te=\ang{-,\ga}$ which was the original definition of attractor invariants \cite{alexandrov_attractor}.
Using the above wall-crossing formulas one can see that, for any central charge \Z, one can express invariants $\bOm_\Z(\ga)$ in terms of invariants $\bOm_*(\ga')$ with $0<\ga'\le \ga$.
Using collections, we can write this statement in the form
\begin{equation}
\bOm_\Z=F_\Z*\bOm_*
\end{equation}
for some collection $F_\Z:S^*\to\bQ$.
Several explicit formulas for such collection
were conjectured in physics literature
\cite{alexandrov_attractor,alexandrov_s}
based on the earlier works on attractor flow trees 
\cite{denef_splita,denef_split,manschot_walla}.
We will prove one of these formulas,
called the attractor tree formula \cite{alexandrov_s,mozgovoy_attractor},
based on the framework of collections developed in this paper.
Actually, we will obtain a one-parameter family of such formulas.
We note that the version of the attractor tree formula that we prove in this paper uses the associative multiplication in the graded algebra \cR.
On the other hand the version conjectured in \cite{alexandrov_attractor},
called a flow tree formula,
can be formulated in terms of the Lie bracket and is of independent interest (see \S\ref{app1}).

\medskip
The paper is organized as follows.
In \S\ref{sec:collections} we introduce collections, plethysm, free construction and various families of collections and their properties such as Harder-Narasimhan collections, exponential and logarithmic collections, geometric collections.
In \S\ref{sec:applications} we discuss applications of collections to Donaldson-Thomas invariants and attractor invariants.
In \S\ref{app1} we introduce Lie collections and, more generally, collections in algebras over operads.
We explain how the flow tree formula can be interpreted in the context of Lie collections.
In Appendix \ref{app2} we explain how our approach can be interpreted in the context of operadic categories.

\medskip
\noindent\textit{Acknowledgments.}
I would like to thank Boris Pioline for explaining to me the subtleties of the flow tree formula and
the attractor tree formula
and for many useful discussions and suggestions.
I am also grateful to Vlad Dotsenko, Jan Manschot, Jos\'e Manuel Moreno, Markus Reineke for useful discussions.

\section{Collections}
\label{sec:collections}
\subsection{Collections and plethysm}
\label{sec:col plethysm}
Let $S$ be a commutative semigroup.
Given a map $f:I\to J$ between finite sets, we consider induced maps
\begin{gather}
f_*:S^I\to S^J,\qquad (f_*\al)_j
=\sum_{i\in f\inv(j)}
\al_i,\qquad \al\in S^I,\ j\in J,\\
f^*:S^J\to S^I,\qquad (f^*\al)_i=\al_{f(i)} 
,\qquad \al\in S^J,\ i\in I.
\end{gather}
For an embedding $f:I\emb J$ and $\al\in S^J$, we will denote $f^*(\al)$ by $\al|_I$.
In what follows we will consider only (non-empty) ordered finite sets $I,J$ and order-preserving maps $f:I\to J$ between them (meaning that $i\le j$ in $I$ implies $f(i)\le f(j)$ in $J$).
Consider the free semigroup $S^*=\bigcup_{n\ge1}S^n$ generated by $S$
and the homomorphism of semigroups
\begin{equation}
S^*\to S,\qquad S^n\ni\al\mto \nn\al=\sum_{i=1}^n\al_i\in S.
\end{equation}

\begin{remark}
In what follows we will assume that the above map has finite fibers, meaning that every element of $S$ can be represented as a sum of elements of $S$ in only finitely many ways. 
This implies in particular that $S$ does not contain the zero element.
Our main example is $S=\bN^r\ms\set0$.
More generally, let $\Ga$ be a free abelian group of finite rank 
and $C\sbs \Ga_\bR=\Ga\ts_\bZ\bR$ be a \idef{polyhedral cone} (meaning a finitely-generated convex cone) that is \idef{strict} (meaning that $C$ does not contain a line).
Then $S=(\Ga\cap C)\ms\set0$ 
satisfies the required condition.
\end{remark}

For any associative ring $\cR$ (not necessarily commutative), we define an abelian group of \idef{collections} in $\cR$
\begin{equation}
\Col_\cR=\Map(S^*,\cR)=\prod_{n\ge1}\Map(S^n,\cR).
\end{equation}
Define the subgroup of $1$-collections $\Cola_\cR=\Map(S,\cR)\sbs\Col_\cR$.

Given two collections $F,G\in\Col_\cR$, we define their \idef{plethysm} $F\circ G\in\Col_\cR$ by
\begin{equation}\label{eq:plethysm}
(F\circ G)(\al)
=\sum_{\pi:I\to J}F(\pi_*\al)\cdot \prod_{j\in J}G(\al|_{\pi\inv j})
,\qquad \al\in S^I,
\end{equation}
where the sum runs over order-preserving maps $\pi:I\to J$ between non-empty finite ordered sets and the product respects the order of $J$.
Here we set $G(\al)=0$ for $\al\in S^0$ and therefore we can assume that $\pi:I\to J$ are surjective.

\begin{lemma}
Let $F,G,H\in\Col_\cR$ be collections such that the values of $G$ are contained in the center of the ring $\cR$.
Then $(F\circ G)\circ H=F\circ(G\circ H)$.
\end{lemma}

\begin{remark}
\label{ident col}
Note that plethysm is right-distributive,
namely 
$$(F+G)\circ H=(F\circ H)+(G\circ H).$$
But it is not left-distributive.
The \idef{identity collection} $\one\in\Col_\cR$ is given by
$$\one(\al)=\begin{cases}
1&n=1,\\
0&n>1,
\end{cases}
\qquad\qquad\al\in S^n.
$$
It satisfies $\one\circ F=F\circ \one=F$.
\end{remark}

\subsection{Action on 1-collections}
Given a collection
$F\in\Col_\cR$ and a $1$-collection $H\in\Cola_\cR$, we have a collection $F\circ H\in \Col_\cR$
\begin{equation}
(F\circ H)(\al)=F(\al)\cdot\prod_{i\in I} H(\al_i),\qquad \al\in S^I,
\end{equation}
where the product is ordered according to the order of $I$.
But we would like to define a related $1$-collection.
Consider the group homomorphism
\begin{equation}
\Col_\cR\to\Cola_\cR,\qquad
F\mto \bar F,\qquad 
\bar F(\ga)=\sum_{\nn\al=\ga}F(\al),\qquad \ga\in S.
\end{equation}
This map is well-defined as the map $\nn-:S^*\to S$ has finite fibers 
by our assumption.
Assume that $\cR$ is an algebra over $\bQ$.
We define the action of $\Col_\bQ$ on $\Cola_\cR$ as follows.
For any $F\in\Col_\bQ$ and $H\in\Cola_\cR$, we define
\begin{equation}
F*H=\ub{F\circ H}\in\Cola_\cR.
\end{equation}
More explicitly, this means
\begin{equation}\label{action2}
(F*H)(\ga)=\sum_{\ov{\al\in S^I}{\nn\al=\ga}}F(\al)\cdot\prod_{i\in I}H(\al_i)
,\qquad \ga\in S,
\end{equation}
where, as usual, the product is ordered according to the order on $I$.

\begin{lemma}
Given collections $F,G\in\Col_\cR$, we have
$$\ub{F\circ \bar G}=\ub{F\circ G}.$$
\end{lemma}
\begin{proof}
For $\ga\in S$, we have
$$\ub{F\circ \bar G}(\ga)
=\sum_{\ov{\al\in S^J}{\nn\al=\ga}}F(\al)\prod_{j\in J}\bar G(\al_j),$$
where
$$\bar G(\al_j)=\sum_{\ov{\be_j\in S^{I_j}}{\nn{\be_j}=\al_j}}G(\be_j).$$
Considering the projection $\pi:I=\sqcup_{j\in J}I_j\to J$ and $\be=(\be_j)_{j\in J}\in S^I$, we obtain $\al=\pi_*\be$, $\nn\be=\nn\al=\ga$ and
$$\ub{F\circ \bar G}(\ga)
=\sum_{\ov{\be\in S^I}{\nn\be=\ga}}
\sum_{\pi:I\to J}F(\pi_*\al)\prod_{j\in J}
G(\be|_{\pi\inv j})=\ub{F\circ G}(\ga).
$$
\end{proof}

\begin{lemma}
Given collections $F,G\in\Col_\bQ$ and $H\in\Cola_\cR$, we have
$$F*(G*H)=(F\circ G)*H.$$
\end{lemma}
\begin{proof}
We have
$F*(G*H)=\ub{F\circ \ub{G\circ H}}
=\ub{F\circ(G\circ H)}
=\ub{(F\circ G)\circ H}
=(F\circ G)*H.
$
\end{proof}

\subsection{Free construction}
\label{sec:trees}
In this section we will construct a map
\begin{equation}
\bT:\Col_\bQ\to\Col_\bQ
\end{equation}
which is reminiscent of the free operad construction in the theory of operads \cite{getzler_operads}.
Note that in the theory of operads, \bT is a monad on the category of symmetric collections (or \bS-objects).
In our case we don't equip \bT with any additional structure.


First, let us recall the notion of a rooted tree.
Define a \idef{(rooted) tree} to be a finite poset $T$ such that
\begin{enumerate}
\item $\max (T)$ has exactly one element. It is called the \idef{root} of $T$ if $\n T>1$.
\item $\forall v\in T\ms\max (T)$, the set $T_{>v}=\sets{u\in T}{u>v}$ has a unique minimal element, denoted $p(v)$ and called the \idef{parent} of $v$.
\end{enumerate}
We define the sets of \idef{leaves} and (internal) \idef{vertices} of $T$ to be respectively
\begin{equation}
L(T)=\min (T),\qquad V(T)=T\ms L(T).
\end{equation}

A tree $T$ having one element is called a \idef{unit tree} (\cf \cite[\S A.1]{fresse_homotopy}).
It satisfies $L(T)=T$, $V(T)=\es$, and it doesn't have a root.
The first axiom implies that any tree $T$ 
(as well as $L(T)$) is non-empty.
If $T$ is not a unit tree, then it has a root, contained in $V(T)$.
For any $v\in T$, define 
its sets of \idef{children} and \idef{leaves} to be respectively
\begin{equation}
\ch(v)=\sets{u\in T}{p(u)=v},\qquad 
L(v)=\sets{u\in L(T)}{u\le v}.
\end{equation}
Note that $\ch(v)$ is empty for leaves and non-empty for internal vertices.

\begin{example}
If $T=\set{v_0}$ is a unit tree, then $T$ has a leaf $v_0$, but doesn't have a root or internal vertices.
Consider a poset $T=(v_0> v_1>\dots>v_n)$ with $n\ge1$.
Then $v_0$ is the root of $T$ and $v_n$ is the only leaf of $T$.
We have $p(v_i)=v_{i-1}$ for $1\le i\le n$.
We also have $L(v_i)=\set{v_n}$ for all $0\le i\le n$.
These trees are the only trees with one leaf. 
\end{example}

We define a \idef{plane tree} to be a rooted tree $T$
such that the set of children $\ch(v)$ is equipped with an order, for all $v\in V(T)$.
Note that in this case $L(T)$ can be equipped 
with a unique order such that the natural surjective map $L(v)\to\ch(v)$ (where $i\mto u$ if $i\in L(u)$) is order-preserving, for all $v\in V(T)$.
This order on $L(T)$ determines the order on $\ch(v)$, for all $v\in V(T)$, uniquely.
We can visualize trees using Hasse diagrams of posets, with the root at the top and leaves at the bottom.

\begin{figure}[ht]
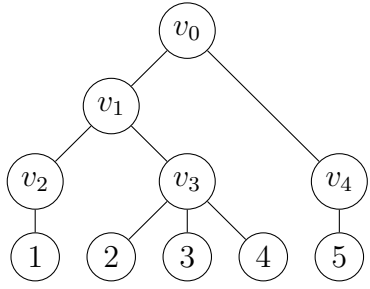

\begin{ctikz}
\node[crc1](0)at(0,0){$v_0$};
\node[crc1](1)at(-1,-1){$v_1$};
\node[crc1](2)at(-2,-2){$v_2$};
\node[crc1](3)at(0,-2){$v_3$};
\node[crc1](4)at(2,-2){$v_4$};
\node[crc1](l1)at(-2,-3){$1$};
\node[crc1](l2)at(-1,-3){$2$};
\node[crc1](l3)at(0,-3){$3$};
\node[crc1](l4)at(1,-3){$4$};
\node[crc1](l5)at(2,-3){$5$};
\draw (0)--(1)--(2) (1)--(3)
(2)--(l1)
(3)--(l2)(3)--(l3)(3)--(l4)
(0)--(4)--(l5);
\end{ctikz}
\caption{A tree $T$ with the root $v_0$, $V(T)=\set{v_0,v_1,v_2,v_3,v_4}$
and $L(T)=\set{1,2,3,4,5}$.}
\end{figure}


Given an ordered set $I$, we define a plane tree on $I$ to be a plane tree $T$ together with an order preserving bijection $L(T)\iso I$.
Let $\cT(I)$ denote the set of all such trees and let $\cT(n)$ denote the set of all such trees for $I=(1<\dots<n)$.

Let $T\in\cT(I)$.
For any $\al\in S^I$ and $v\in T$, define \begin{equation}\label{tuple on children}
\al|_{\ch v}=(\al_u)_{u\in\ch v}\in S^{\ch v},
\qquad
\al_u=\sum_{i\in L(u)}\al_i\in S.
\end{equation}
For any collection $F\in\Col_\bQ$, define
\begin{equation}
F(T):S^I\to\bQ,\qquad
F(T)(\al)=\prod_{v\in V(T)}F(\al|_{\ch v}),
\qquad \al\in S^I.
\end{equation}
Note that for the unit tree $T$ (satisfying $V(T)=\es$), we have $F(T)(\al)=1$ for all $\al\in S$.
This is the identity collection by Remark
\ref{ident col}. 
 
Assume that $F(\al)=0$ for all $\al\in S$ (we say that $F$ is \idef{supported in cardinality} $\ge2$).
Then we define a new collection $\bT F\in\Col_\bQ$ by
the rule
\begin{equation}
(\bT F)(\al)=\sum_{T\in\cT(I)}F(T)(\al)
=\sum_{T\in\cT(I)}\prod_{v\in V(T)}F(\al|_{\ch v}),
\qquad \al\in S^I.
\end{equation}
Note that the product above is zero if some vertex $v\in V(T)$ has only one child.
Therefore we consider only trees with internal vertices having $\ge 2$ children and there are only finitely many such trees in $\cT(I)$.
Note that $(\bT F)(\al)=1$ for all $\al\in S$
and we can represent 
\begin{equation}
\bT F=\one+\bT_{\ge2}F,
\end{equation}
where $\one$ is the identity collection and $\bT_{\ge2}F$ is supported in cardinality $\ge2$.

\begin{lemma}
For any collection $F$ supported in cardinality $\ge2$,
we have
\begin{equation}
\bT F=\one+F\circ\bT F.
\end{equation}
\end{lemma}
\begin{proof}
According to the definition of plethysm, the value of $(F\circ\bT F)(\al)$, for $\al\in S^I$, is obtained by the sum over all trees in $\cT(I)$, except possibly the unit tree (for $\n I=1$), because all trees that we obtain will have a root.
This implies the statement.
\end{proof}

\begin{remark}
In the context of symmetric collections and the free operad construction, we have an equation $\bT F=\one\oplus (F\circ \bT F)$ (see \eg \cite[\S5.4]{loday_algebraic}).
A similar result can be formulated for collections over operadic categories (see Appendix \ref{app2}).
\end{remark}

\begin{theorem}
\label{thm:inverse}
For any collection $F$ supported in cardinality $\ge2$,
we have
\begin{equation}
(\one-F)\circ\bT F
=\bT F\circ (\one-F)
=\one.
\end{equation}
\end{theorem}
\begin{proof}
Applying the previous lemma and using the right-distributivity, we obtain $(\one-F)\circ\bT F=\one$.
This implies that $\one-F$ has a right inverse.
In particular, $\bT F=\one+\bT_{\ge2}F$ has a right inverse and we conclude that this inverse is equal to $\one-\bF$.
\end{proof}

\begin{remark}
\label{rm:inverse}
Let $F$ be a collection such that $F(\al)=1$ for all $\al\in S$.
Then $\one-F$ is supported in cardinality $\ge2$ and we obtain
from the above theorem that $F$ is invertible and
\begin{equation}
F\inv=\bT(\one-F).
\end{equation}
\end{remark}

\subsection{HN collections}
\label{sec:HN}
Let $S$ be a commutative semigroup as before.
We define a \idef{(weak) stability} $\le$ 
on $S$ to be a pre-order
(meaning a transitive binary relation with any two elements comparable)
such that for any $a,b\in S$ we have
\begin{equation}
a\le a+b\le b
\qquad\text{or}\qquad
b\le a+b\le a.
\end{equation}
We will write $a< b$ 
if $a\le b$ and $b\not\le a$
and we will write $a\sim b$ 
if $a\le b$ and $b\le a$.

\begin{example}
\label{ex:stab}
Let $\Z:\bZ^n\to\bC$ be a linear map of the form
$\Z=-\te+\bi\rho$
\begin{equation}
\Z(a)
=-\sum_k\te_ka_k+\bi\sum_k\rho_ka_k,\qquad
a\in\bZ^n,
\end{equation}
for some $\te,\,\rho\in\Hom(\bZ^n,\bR)\iso\bR^n$ with $\rho_k>0$ for $1\le k\le n$.
Then we define a relation $\le_\Z$ on $S=\bN^n\ms\set0$ by
\begin{equation}
a\le_\Z b\qquad\text{if}\qquad \Arg\Z(a)\le\Arg\Z(b),
\end{equation}
where $\Arg(re^{\bi\vi})=\vi$ for $r>0$ and $\vi\in(-\pi,\pi]$.
Equivalently, $a\le_\Z b$ if $\mu_\Z(a)\le\mu_\Z(b)$, where
\begin{equation}
\mu_\Z(a)=\frac{\sum_k \te_ka_k}{\sum_k \rho_ka_k}.
\end{equation}
This is a weak stability on $S$.
We will write $a\sim_\Z b$ if $a\le_\Z b$ and $b\le_\Z a$ or, equivalently, $\mu_\Z(a)=\mu_\Z(b)$.
Usually we will fix just $\te\in\bR^n$ and assume that $\rho=(1,\dots,1)$.
Then
\begin{equation}
\mu_\Z(a)=\mu_\te(a)=\frac{\sum_k \te_ka_k}{\sum_k a_k}
\end{equation}
and we write $a\le_\te b$ and $a\sim_\te b$ instead of $a\le_\Z b$ and $a\sim_\Z b$, respectively.
\end{example}

\begin{remark}\label{generic stab}
We will say that elements $a,b\in S$ are \idef{proportional}, and write $a\parallel b$, if $ma=nb$ for some $m,n\ge 1$.
For any $a\in S$, we have $a\le a+a\le a$, hence $a\sim 2a$ and more generally $na\sim a$ for $n\ge1$.
Therefore $a\prl b$ implies $a\sim b$.
We will say that stability $\le$ is \idef{generic} if $a\sim b$ implies $a\prl b$.
\end{remark}

Given a stability $\le$ on $S$,
we define the \idef{Harder-Narasimhan (HN) collection}
$s\in\Col_\bQ$
\begin{equation}\label{HN col}
s(\al)
=\begin{cases}
1&\al_1>\dots>\al_n,\\
0&\text{otherwise},
\end{cases}\qquad\qquad\al\in S^n.
\end{equation}
For a weak stability of the form $\le_\Z$ (respectively $\le_\te$) from Example \ref{ex:stab}, we denote the corresponding HN collection by $s_\Z$ (respectively $s_\te$).
Note that $s(\al)=1$ for $\al\in S$ and therefore $s\in\Col_\bQ$ is invertible by Remark \ref{rm:inverse} and
\begin{equation}
s\inv=\bT(\one-s),
\end{equation}
where $\one-s$ is supported in cardinality $\ge2$.
The following more explicit formula for $s\inv$ can be deduced from the results of \cite{reineke_harder-narasimhan}.

\begin{theorem}
\label{th:R}
For any $\al\in S^n$ and $1\le k<n$, let 
$\al_{\le k}=\al_1+\dots+\al_k$
and $\al_{>k}=\al_{k+1}+\dots+\al_n.$
Then
\begin{equation}\label{R col}
s\inv(\al)
=\begin{cases}
(-1)^{n-1}&\al_{\le k}>\al_{>k}\ \forall\, 1\le k<n,\\
0&\text{otherwise}.
\end{cases}
\end{equation} 
\end{theorem}
\begin{proof}
For any surjective, order-preserving map 
$$\pi:I=\set{1,\dots,n}\to J=\set{1,\dots,m},$$ define
$$I'=\set{k_1,\dots,k_{m-1}}\sbs I,\qquad k_j=\max\pi\inv(j),
\qquad j=1,\dots,m-1.$$
Note that $\max\pi\inv(m)=n$ and the subset $I'\sbs I\ms\set n$ determines the map $\pi$ uniquely.
Conversely, for any subset $I'\sbs I\ms\set n$, we can construct an order-preserving map $\pi$ such that $I'$ 
is obtained from it by the above construction.

Let $s'$ denote the above collection \eqref{R col}
and $\al\in S^I$.
The summand of $(s'\circ s)(\al)$ corresponding to $\pi:I\to J$ is
equal to
$$s'(\pi_*\al)\cdot\prod_{j\in J}s(\al|_{\pi\inv j}).$$
Define
\begin{equation*}
I_1=\sets{1\le k<n}{\al_{\le k}>\al_{>k}},\qquad
I_2=\sets{1\le k<n}{\al_{k}\le\al_{k+1}}.
\end{equation*}
The first factor is nonzero if and only if $\al_{\le k_j}>\al_{>k_j}$ for all $1\le j<m$,
hence $I'\sbs I_1$.
The second factor is nonzero if and only if $\al_{k}>\al_{k+1}$ for all $k\notin I'$, 
hence $I_2\sbs I'$.
We conclude that $(s'\circ s)(\al)=0$ if $I_2\not\sbs I_1$.
If $I_2\sbs I_1$, then
$$(s'\circ s)(\al)=\sum_{I_2\sbs I'\sbs I_1}(-1)^{\n {I'}}
=(-1)^{\n {I_2}}\cdot(1-1)^{\n{I_1\ms I_2}}
=\begin{cases}
(-1)^{\n {I_2}}&I_1=I_2,\\
0&I_1\ne I_2.
\end{cases}
$$
We claim that if $I_1=I_2$, then $I_1=I_2=\es$.
Assume that $I_1=I_2\ne\es$ and let $k=\min I_1$
and 
$$a=\al_{<k},\qquad b=\al_k,\qquad
c=\al_{>k}.$$
If $k>1$, then
$$a\le b+c,\qquad a+b>c,\qquad a>b.$$
Therefore $a>a+b>c$, hence $a>b+c$, a contradiction.
Let $k\ge1$ be minimal such that $k\notin I_1$ and let 
$a,b,c$ be defined as before.
If $k\ne n$, then
$$a>b+c,\qquad a+b\le c,\qquad
a\le b$$
and this again leads to a contradiction.
We conclude that $I_1=I_2=\set{1,\dots,n-1}$.
But then $\al_1\le\al_2\le\dots\le\al_n$ and this implies $\al_1\le\al_{>1}$, a contradiction.
Therefore $I_1=I_2=\es$.
If $n>1$, then $\al_1>\al_2>\dots>\al_n$, hence $\al_1>\al_{>1}$ and $I_1\ne\es$, a contradiction.
Therefore, if $I_1=I_2$, then $n=1$.
This implies that $(s'\circ s)(\al)=1$ for $\al\in S^n$ with $n=1$ and $(s'\circ s)(\al)=0$ otherwise.
Therefore $s'\circ s=\one$.
\end{proof}

Let us assume now that we have two stabilities $\le_\te$ and $\le_{\te'}$ on $S$ and
let $s_\te$ and $s_{\te'}$ be the corresponding HN collections.
Assume that we have $1$-collections $H,H'$ related by the equation
\begin{equation}
s_{\te}*H=s_{\te'}*H'.
\end{equation}
Then we can express $H'$ as
\begin{equation}
H'=s_{\te\to\te'}*H,\qquad
s_{\te\to\te'}=s_{\te'}\inv\circ s_\te
\end{equation}
and it is useful to know an explicit formula 
for the collection $s_{\te'}\inv\circ s_\te$.
The following formula for this collection can be deduced from the results of \cite{joyce_configurations}.

\begin{theorem}
\label{th:J}
For any $\al\in S^n$ and $1\le k<n$, define
$$
\eps_k(\al)=
\begin{cases}
-1&\al_k\le_\te\al_{k+1}\text{ and } \al_{\le k}>_{\te'}\al_{>k}\\
1&\al_k>_\te\al_{k+1}\text{ and } \al_{\le k}\le_{\te'}\al_{>k}\\
0&\text{otherwise}
\end{cases}
$$
Then the collection
$s_{\te\to\te'}=s_{\te'}\inv\circ s_\te$
satisfies $s_{\te\to\te'}(\al)=\prod_{k=1}^{n-1}\eps_k(\al)$. 
\end{theorem}

\begin{proof}
As in Theorem \ref{th:R} we identify a surjective order-preserving map $\pi:I=\set{1,\dots,n}\to J$
with a subset $I'\sbs I\ms\set n$ such that $\n{I'}=\n J-1$.
The summand of $(s_{\te'}\inv\circ s_\te)(\al)$ corresponding to $\pi:I\to J$ is
equal to
$$s_{\te'}\inv(\pi_*\al)\cdot\prod_{j\in J}s_\te(\al|_{\pi\inv j}).$$
Define
\begin{equation*}
I_1=\sets{1\le k<n}{\al_{\le k}>_{\te'}\al_{>k}},\qquad
I_2=\sets{1\le k<n}{\al_{k}\le_\te\al_{k+1}}.
\end{equation*}
As in Theorem \ref{th:R},
the first factor is nonzero if and only if $I'\sbs I_1$ and the second factor is nonzero if and only if $I_2\sbs I'$.
We conclude that 
$(s_{\te'}\inv\circ s_\te)(\al)=0$ if $I_2\not\sbs I_1$.
If $I_2\sbs I_1$, then
$$(s_{\te'}\inv\circ s_\te)(\al)=\sum_{I_2\sbs I'\sbs I_1}(-1)^{\n {I'}}
=(-1)^{\n {I_2}}\cdot(1-1)^{\n{I_1\ms I_2}}
=\begin{cases}
(-1)^{\n {I_2}}&I_1=I_2,\\
0&I_1\ne I_2.
\end{cases}
$$

Condition $I_1=I_2$ means that for all $1\le k<n$ we have
\begin{enumerate}
\item If $\al_k\le_\te\al_{k+1}$ then $\al_{\le k}>_{\te'}\al_{>k}$.
\item If $\al_k>_\te\al_{k+1}$ then $\al_{\le k}\le_{\te'}\al_{>k}$.
\end{enumerate}
The number of $k$ satisfying the first condition is equal to $\n {I_1}=\n {I_2}$. 
This implies that
$(s_{\te'}\inv\circ s_\te)(\al)
=\prod_{k=1}^{n-1}\eps_k(\al)$.
\end{proof}

\subsection{Exponential and logarithmic collections}
\label{sec:exp-log}
Assume that we have a $1$-collection $H:S\to \bQ$ and we want to define a new $1$-collection $H':S\to\bQ$ such that
\begin{equation}\label{exp1}
1+\sum_{\ga\in S}H'(\ga)x^\ga
=\exp\rbr{\sum_{\ga\in S}H(\ga)x^\ga}.
\end{equation}
It is clear that $H'$ can be obtained from $H$ by the action of a collection.

\begin{lemma}\label{lm:exp}
Equation \eqref{exp1} is equivalent to
\begin{equation}
H'=\cexp*H,
\end{equation}
where the 
\idef{exponential collection} $\cexp$ is defined by
\begin{equation}
\cexp(\al)=\frac1{n!},\qquad
\al\in S^n.
\end{equation}
\end{lemma}
\begin{proof}
Note that the set $S^n$ is equipped with an action of the permutation group $\fS_n$ and the orbits are parametrized by maps $m:S\to\bN$ with $\sum_{\ga\in S}m(\ga)=n$.
The stabilizer of an element in such orbit has cardinality $\prod_{\ga\in S} m(\ga)!$.
We obtain
\begin{multline*}
1+\sum_{\ga\in S}(\cexp* H)(\ga)x^\ga
=\sum_{\ov{\al\in S^n}{n\ge0}}\frac1{n!}
\prod_{i=1}^n H(\al_i)x^{\al_i}
=\sum_{m:S\to\bN}\prod_{\ga\in S}\frac1{m(\ga)!}(H(\ga)x^\ga)^{m(\ga)}\\
=\prod_{\ga\in S}\rbr{\sum_{m\ge0}\frac1{m!}(H(\ga)x^\ga)^m}
=\prod_{\ga\in S}\exp\rbr{H(\ga)x^\ga}
=\exp\rbr{\sum_{\ga\in S}H(\ga)x^\ga}.
\end{multline*}
\end{proof}
We define the \idef{logarithmic collection} $\clog$
by
\begin{equation}
\clog(\al)=\frac{(-1)^{n-1}}{n},\qquad \al\in S^n.
\end{equation}

\begin{lemma}\label{lm:exp-log}
We have $\cexp\circ\clog=\one$.
\end{lemma}
\begin{proof}
Let $\al\in S^n$.
Given a surjective order-preserving map $\pi:\set{1,\dots,n}\to \set{1,\dots,m}$,
let $k_j=\n{\pi\inv j}$ for $1\le j\le m$.
Then the summand of $(\clog\circ\cexp)(\al)$
corresponding to $\pi$ is equal
$$\frac{(-1)^{m-1}}{m}\prod_{j=1}^m\frac{1}{k_j!}.$$
The corresponding generating function (over all $m,n$ and maps $\pi$ as above) is
\begin{multline*}
\sum_{m\ge1}\sum_{k_1,\dots,k_m\ge1}
\frac{(-1)^{m-1}}{m}\prod_{j=1}^m\frac{x^{k_j}}{k_j!}
=\sum_{m\ge1}\frac{(-1)^{m-1}}{m}
\prod_{j=1}^m\rbr{\sum_{k\ge1}\frac{x^{k}}{k!}}\\
=\sum_{m\ge1}\frac{(-1)^{m-1}}{m}
(e^x-1)^m
=\log(1+e^x-1)=x.
\end{multline*}
This implies that $(\clog\circ\cexp)(\al)=1$ for $n=1$ and is zero for $n>1$,
hence $\clog\circ\cexp=\one$.
\end{proof}

For the action on $1$-collections in a non-commutative algebra $\cR$ we have to be more careful, 
because in the proof of Lemma \ref{lm:exp} we relied on the fact that $H(\ga)$ commute with each other for different $\ga$.
Let us assume that we have a graded algebra $\cR=\bop_{\ga\in\bN^r}\cR_\ga$ such that $[\cR_\ga,\cR_{\ga'}]=0$ for $\ga\prl\ga'$.
Then the following versions of the exponential and logarithmic collections will be useful
\begin{equation}
\cexp_\prl\in\Col_\bQ,\qquad
\cexp_\prl(\al)=
\begin{cases}
\frac1{n!}&\al_1\prl\dots\prl\al_n,\\
0&\text{otherwise},
\end{cases}
\qquad \al\in S^n,
\end{equation}
\begin{equation}
\clog_\prl\in\Col_\bQ,\qquad
\clog_\prl(\al)=
\begin{cases}
\frac{(-1)^{n-1}}{n}&\al_1\prl\dots\prl\al_n,\\
0&\text{otherwise},
\end{cases}
\qquad \al\in S^n.
\end{equation}
They are again inverse to each other
\begin{equation}
\cexp_\prl\circ\clog_\prl
=\clog_\prl\circ\cexp_\prl
=\one.
\end{equation}

More generally, assume that we have a weak stability $\le_\te$.
We will write $\ga\sim_\te\ga'$ if $\ga\le_\te\ga'$ and $\ga'\le_\te\ga$.
Then we define the corresponding exponential and logarithmic collections
\begin{equation}
\cexp_\te(\al)=\begin{cases}
\frac1{n!}&\al_1\sim_\te\dots\sim_\te\al_n,\\
0&\text{otherwise},
\end{cases}
\qquad \al\in S^n,
\end{equation}
\begin{equation}
\clog_\te(\al)=\begin{cases}
\frac{(-1)^{n-1}}{n}&\al_1\sim_\te\dots\sim_\te\al_n,\\
0&\text{otherwise},
\end{cases}
,\qquad \al\in S^n,
\end{equation}
which are again inverse to each other.
Note that if $\le_\te$ is a generic stability (see Remark \ref{generic stab}), then
\begin{equation}
\cexp_\te=\cexp_\prl,\qquad\clog_\te=\clog_\prl.
\end{equation}

\begin{theorem}\label{th:log-s}
Given a weak stability $\le_\te$, define collection 
\begin{equation}
g_\te=\clog_\te\circ s_\te\inv.
\end{equation}
Then, for any $\al\in S^n$,
\begin{equation}
g_{\te}(\al)=\begin{cases}
\frac{(-1)^{n-1}}{n_0+1}& \al_{\le k}\ge_\te\al_{>k}
\ \forall\, 1\le k<n,\\
0&\text{otherwise,}
\end{cases}
\end{equation}
where $n_0=n_0(\al)$ is the number of $1\le k<n$ such that $\al_{\le k}\sim_\te \al_{>k}$.
\end{theorem}
\begin{proof}
As in Theorem \ref{th:R} we identify a surjective order-preserving map $\pi:I=\set{1,\dots,n}\to J$
with a subset $I'\sbs I\ms\set n$ such that $\n{I'}=\n J-1$.
The summand of $(\clog_\te\circ s_\te\inv)(\al)$ corresponding to $\pi:I\to J$ is
equal to
$$\clog_\te(\pi_*\al)\cdot\prod_{j\in J}s_\te\inv(\al|_{\pi\inv j})
.$$
Define
\begin{equation*}
I_1=\sets{1\le k<n}{\al_{\le k}\sim_{\te}\al_{>k}}
,\qquad
I_2=\sets{1\le k<n}{\al_{\le k}\le_\te\al_{>k}}.
\end{equation*}
As in Theorem \ref{th:R},
the first factor is nonzero if and only if $I'\sbs I_1$,
in which case its value is $(-1)^{\n J-1}/\n J$.
The second factor is nonzero if and only if for all $k\notin I'$ we have $\al_{\le k}>_\te\al_{>k}$,
hence $I_2\sbs I'$.
The requirement $I_2\sbs I'\sbs I_1$ implies that $I_1=I_2$, hence $\al_{\le k}\ge_\te \al_{>k}$ for all $1\le k<n$.
We conclude that the value of $(\clog_\te\circ s_\te)(\al)$ is equal to (with $n_0=\n {I_1}=\n{I'}$)
$$\frac{(-1)^{\n{I'}}}{\n{I'}+1}
(-1)^{n-\n I'-1}
=\frac{(-1)^{n-1}}{n_0+1}
$$
if $\al_{\le k}\ge_\te \al_{>k}$ for all $1\le k<n$ and zero otherwise.
\end{proof}

\begin{theorem}\label{th:s-exp}
Let 
$g_\te\inv=s_\te\circ\cexp_\te$.
Then, for any $\al\in S^n$, we have
\begin{equation}
g_\te\inv(\al)=
\begin{cases}
\prod_{i=1}^m\frac1{(k_i-k_{i-1})!}&\al_{k}\ge_\te\al_{k+1}\ \forall\,1\le k<n\\
0&\text{otherwise}
\end{cases}
\end{equation}
where $0=k_0<k_1<\dots<k_m=n$ is the set of all $k$ such that $\al_k>_\te\al_{k+1}$.
\end{theorem}
\begin{proof}
As in Theorem \ref{th:R} we identify a surjective order-preserving map $\pi:I=\set{1,\dots,n}\to J$
with a subset $I'\sbs I\ms\set n$ such that $\n{I'}=\n J-1$.
The summand of $(s_\te\circ\cexp_\te)(\al)$ corresponding to $\pi:I\to J$ is
equal to
$$s_\te(\pi_*\al)\cdot
\prod_{j\in J}\cexp_\te(\al|_{\pi\inv j})$$
The second factor is nonzero if and only if
for all $k\notin I'$ we have
$\al_k\sim_\te\al_{k+1}$.
Then the first factor is nonzero if and only if
for all $k\in I'$ we have $\al_k>_\te\al_{k+1}$.
We conclude that we should have $\al_k\ge_\te\al_{k+1}$ for all $1\le k<n$ and that $I'$ is uniquely determined as 
$I'=\sets{1\le k<n}{\al_k>_\te\al_{k+1}}$.
The statement of the theorem now follows from the definition of $s_\te$ and $\cexp_\te$. 
\end{proof}

\subsection{Geometric collections}
\label{sec:geom col}
In this section we will prove a generalization of Theorem \ref{th:log-s}.
For any $t\in\bQ$, define a \idef{geometric collection} $\si_t$ by
\begin{equation}
\si_t(\al)=t^{n-1},\qquad \al\in S^n.
\end{equation}

\begin{remark}
Note that for $t=0$, we obtain $\si_t(\al)=1$ if $\al\in S^n$ with $n=1$ and zero otherwise.
Therefore we have $\si_0=\one$.
\end{remark}

\begin{lemma}
Let $\le_\te$ be a weak stability and $t\in\bQ$.
For any $\al\in S^n$, we have
$$(s_\te\inv\circ \si_t)(\al)
=t^{n-1-n_+}(t-1)^{n_+}
,$$
where $n_+=n_+(\al)$ is the number of $1\le k<n$ such that $\al_{\le k}>_\te\al_{>k}$.
\end{lemma}
\begin{proof}
As in Theorem \ref{th:R} we identify a surjective order-preserving map $\pi:I=\set{1,\dots,n}\to J$
with a subset $I'\sbs I\ms\set n$ such that $\n{I'}=\n J-1$.
The summand of $(s_{\te}\inv\circ \si_t)(\al)$ corresponding to $\pi:I\to J$ is
equal to
$$s_{\te}\inv(\pi_*\al)\cdot\prod_{j\in J}\si_t(\al|_{\pi\inv j})
=s_{\te}\inv(\pi_*\al)\cdot t^{n-\n J}
.$$
Define
\begin{equation*}
I_1=\sets{1\le k<n}{\al_{\le k}>_{\te}\al_{>k}}.
\end{equation*}
As in Theorem \ref{th:R},
the first factor is nonzero if and only if $I'\sbs I_1$,
in which case its value is $(-1)^{\n J-1}$.
We conclude that
$$(s_{\te}\inv\circ \si_t)(\al)
=\sum_{I'\sbs I_1}(-1)^{\n{I'}}t^{n-\n{I'}-1}
=t^{n-1}(1-t\inv )^{\n{I_1}} 
$$
and the theorem follows.
\end{proof}

\begin{theorem}
\label{th:g-te-t}
Given a weak stability $\le_\te$ and $t\in\bQ$,
define collection
\begin{equation}\label{eq:g-te-t1}
g_{\te,t}
=g_\te\circ\si_t
=\clog_\te\circ s_\te\inv\circ \si_t.
\end{equation}
Then, for any $\al\in S^n$, we have
\begin{equation}\label{eq:g-te-t2}
g_{\te,t}(\al)
=t^{n_-}(t-1)^{n_+}\frac{t^{n_0+1}-(t-1)^{n_0+1}}{n_0+1}
\end{equation}
where $n_+=n_+(\al)$ (respectively $n_-$ and $n_0$) is the number of $1\le k<n$ such that $\al_{\le k}>_\te\al_{>k}$ 
(respectively $<_\te$ and $\sim_\te$).
\end{theorem}

\begin{proof}
As in Theorem \ref{th:R} we identify a surjective order-preserving map $\pi:I=\set{1,\dots,n}\to J$
with a subset $I'\sbs I\ms\set n$ such that $\n{I'}=\n J-1$.
Let $f=s_\te\inv\circ \si_t$ (determined in the previous lemma).
The summand of $(\clog_\te\circ f)(\al)$ corresponding to $\pi:I\to J$ is
equal to
$$\clog_\te(\pi_*\al)\cdot\prod_{j\in J}f(\al|_{\pi\inv j})
.$$
Define
\begin{equation*}
I_1=\sets{1\le k<n}{\al_{\le k}\sim_{\te}\al_{>k}}.
\end{equation*}
As in Theorem \ref{th:R},
the first factor is nonzero if and only if $I'\sbs I_1$,
in which case its value is $(-1)^{\n J-1}/\n J$.
Let $r_j$ be the number of inequalities $\al_{\le k}>_\te\al_{>k}$ for $\al|_{\pi\inv j}$.
Then the number of such inequalities for the whole $\al$ is $n_+=\sum_{j\in J}r_j$.
We conclude that the value of $(\clog_\te\circ f)(\al)$ is equal to
(with $n_0=\n {I_1}$ and $n-1=n_0+n_++n_-$)
\begin{multline*}
\sum_{I'\sbs I_1}\frac{(-1)^{\n{I'}}}{\n{I'}+1}
t^{n-\n{I'}-1}\prod_j(1-t\inv)^{r_j}
=t^{n-1}(1-t\inv)^{n_+}\sum_{k=0}^{n_0}\frac1{k+1}
\binom {n_0}k(-t\inv)^k\\
=t^{n-1}(1-t\inv)^{n_+}\frac{-t}{n_0+1}
\rbr{\sum_{k=0}^{n_0+1}\binom{n_0+1}k(-t\inv)^k-1}
=t^n(1-t\inv)^{n_+}\frac{1-(1-t\inv)^{n_0+1}}{n_0+1}
\end{multline*}
which simplifies to \eqref{eq:g-te-t2}.
\end{proof}

\begin{remark}\br
\label{rm:t=1/2}
\begin{enumerate}
\item 
For $t=0$, we obtain that $g_{\te,0}(\al)$ can be nonzero only if $n_-=0$, hence $\al_{\le k}\ge_\te\al_{>k}$ for all $1\le k<n$.
In this case 
$$g_{\te,0}(\al)=(-1)^{n_+}\frac{(-1)^{n_0}}{n_0+1}
=\frac{(-1)^{n-1}}{n_0+1}.$$
This is the statement of Theorem \ref{th:log-s}
as $g_{\te,0}=g_\te\circ\si_0=g_\te$.
\item
For $t=1$, we obtain that $g_{\te,1}(\al)$ can be nonzero only if $n_+=0$, hence $\al_{\le k}\le_\te\al_{>k}$ for all $1\le k<n$.
In this case 
$$g_{\te,1}(\al)=\frac{1}{n_0+1}.$$
\item
For $t=\oh$, we obtain
\begin{equation*}
g_{\te,\oh}(\al)=\frac{(-1)^{n_+}}{2^n}\cdot
\frac{1+(-1)^{n_0}}{n_0+1}.
\end{equation*}
Note that it is zero for odd $n_0$.
\end{enumerate}
\end{remark}

\section{Applications to DT invariants}
\label{sec:applications}
\subsection{DT invariants}
\label{sec:DT}
Consider a graded algebra $\cR=\bop_{\ga\in\bN^r}\cR_\ga$
over \bQ and its completion
\begin{equation}
\what\cR=\ilim_n\cR/F_n\cR,\qquad F_n\cR=\bop_{\de\cdot\ga\ge n}\cR_\ga,\qquad
\de=(1,\dots,1).
\end{equation}
We further define a pro-nilpotent Lie algebra and the corresponding pro-unipotent group
\begin{equation}
\fg=\ilim_n F_1\cR/F_n\cR\sbs\what\cR,\qquad
G=\exp(\fg)=1+\fg\sbs\what\cR.
\end{equation}

We always assume that $\ga\prl\ga'\imp[\cR_\ga,\cR_{\ga'}]=0$.
Let $S=\bN^r\ms\set0$.
We will consider only 1-collections $A:S\to\cR$ such that $A(\ga)\in\cR_\ga$ for $\ga\in S$.
Given such collection, we can associate with it the series
\begin{equation}
\what A=1+\sum_{\ga\in S}A(\ga)\in G.
\end{equation}

Let us fix a 1-collection $A:S\to \cR$.
As in the introduction, given a central charge \Z (or any weak stability $\le_\Z$ on $S$, see Example \ref{ex:stab}),
we define a $1$-collection $A_\Z$ of \qq{stacky DT invariants} by the formula 
\begin{equation}\label{wc1}
A(\ga)=\sum_{\ov{\al_1+\dots+\al_n=\ga}
{\al_1>_\Z\dots>_\Z\al_n}}
A_\Z(\al_1)\dots A_\Z(\al_n),\qquad \ga\in S.
\end{equation}
We call this equation the \idef{basic wall-crossing} formula.
Using the HN collection $s_\Z$ \eqref{HN col} and its inverse $s_\Z\inv$ \eqref{R col}, we can write
\begin{equation}
A=s_\Z*A_\Z,\qquad A_\Z=s_\Z\inv*A.
\end{equation}

\begin{example}\label{ex:quiver}
Let $Q=(Q_0,Q_1,s,t)$ be a quiver, where $s,t:Q_1\to Q_0$ are the source and the target maps respectively.
For any $\ga\in\bN^r$, $r=\n{Q_0}$, let $R(Q,\ga)=\bop_{a:i\to j}\Hom(\bC^{\ga_i},\bC^{\ga_j})$ be the space of quiver representations having dimension vector $\ga$.
It is equipped with an action of the group $G_\ga=\prod_i\GL_{\ga_i}(\bC)$ so that the orbits correspond to isomorphism classes of representations.
Given a central charge $\Z$, let $R_\Z(Q,\ga)\sbs R(Q,\ga)$ be the open subset of $\Z$-semistable representations.
We can organize invariants of these spaces into a $1$-collection $A_\Z$ as follows.
We equip $\Ga=\bZ^r$ with a skew-symmetric form $\ang{\al,\be}=\hi(\al,\be)-\hi(\be,\al)$, where
$\hi(\al,\be)=\sum_i\al_i\be_i-\sum_{a:i\to j}\al_i\be_j$ is the Euler form of the quiver $Q$.
Then we consider the algebra 
$\cR=\bop_{\ga\in\bN^r}\cR_\ga=\bop_{\ga\in\bN^r}\bQ(y)x^\ga$ equipped with the product
\begin{equation}
x^\ga\cdot x^{\ga'}=(-y)^{\ang{\ga,\ga'}}x^{\ga+\ga'}.
\end{equation}
Define 1-collections
\begin{gather}
A(\ga)=(-y)^{\hi(\ga,\ga)}\frac{P(R(Q,\ga);y)}{P(G_\ga;y)}x^\ga\in\cR_\ga,\\
A_\Z(\ga)=(-y)^{\hi(\ga,\ga)}\frac{P(R_\Z(Q,\ga);y)}{P(G_\ga;y)}x^\ga\in\cR_\ga,
\end{gather}
where $P(X;y)$ is the virtual Poincar\'e polynomial, additive on complements and defined by $P(X;y)=\sum_n\dim H^n(X,\bQ)\cdot (-y)^n$, for a smooth projective variety $X$.
The fact that these invariants satisfy \eqref{wc1} follows from the uniqueness of the Harder-Narasimhan filtration \cite{reineke_harder-narasimhan,joyce_configurationsa,kontsevich_stability}.
This is the simplest example of $1$-collections related by \eqref{wc1}, but many other examples follow the same paradigm \cite{joyce_configurations,kontsevich_stability,joyce_theory,bridgeland_introduction}, albeit in different contexts.
For us, however, the context is quite irrelevant.
We start with a 1-collection $A$ and then define 1-collections $A_\Z$ using \eqref{wc1} recursively.
\end{example}

We define a 1-collection $\bOm_\Z$ of \qq{rational DT invariants} by the formula
\begin{equation}\label{DT1}
\bOm_\Z=\clog_\Z*\hOm_\Z
=(\clog_\Z\circ s_\Z\inv)*\hOm
=g_\Z*\hOm,
\end{equation}
where collection $g_\Z=\clog_\Z\circ s_\Z\inv$ was determined explicitly in Theorem \ref{th:log-s}.
For any ray $\ell\sbs\bC$, consider the series
\begin{equation}
\hOm_{\Z,\ell}=1+\sum_{Z(\ga)\in\ell}\hOm_\Z(\ga)\in G,
\qquad
\bOm_{\Z,\ell}=\sum_{Z(\ga)\in\ell}\bOm_\Z(\ga)\in \fg.
\end{equation}
We will say that $\Z$ is \idef{weakly generic} if 
$\ga\sim_\Z\ga'$ (or, equivalently, $\Z(\ga),\,\Z(\ga')$ are contained in the same ray)
implies $[\cR_\ga,\cR_{\ga'}]=0$.
In this case we obtain from Lemma \ref{lm:exp} that
\begin{equation}\label{series equation}
\bOm_{\Z,\ell}=\log(\hOm_{\Z,\ell}).
\end{equation}
We note that the meaning and significance of $\bOm_\Z(\ga)$ for non weakly generic \Z is not quite clear.
If $\Z$ is generic (see Remark \ref{generic stab}), then it is weakly generic
and \eqref{series equation} is satisfied.
Note that in this case we have
\begin{equation}\label{eq:from A to bar A}
\bOm_\Z=\clog_\prl*\hOm_\Z.
\end{equation}

Given two central charges $\Z,\,\Z'$ we have
$A=g_\Z\inv* \bOm_\Z=g_{\Z'}\inv*\bOm_{\Z'}$, hence
we have a wall-crossing formula
\begin{equation}\label{bom Z to Z'}
\bOm_{\Z'}=(g_{\Z'}\circ g_Z\inv)*\bOm_\Z
=(\clog_{\Z'}\circ s_{\Z\to\Z'}\circ\cexp_\Z)*\bOm_Z,
\end{equation}
where $s_{\Z\to\Z'}=s_{\Z'}\inv\circ s_\Z$ (see Theorem \ref{th:J}).
An explicit formula for the collection $g_{\Z'}\circ g_Z\inv$ can be obtained from Theorem \ref{th:log-s} and Theorem \ref{th:s-exp}
(\cf \cite{joyce_configurations}).
Alternatively, we can write this collection using trees as $g_{\Z'}\circ\bT(\one-g_\Z)$.
For generic central charges we obtain
\begin{equation}\label{bom Z to Z' 2}
\bOm_{\Z'}
=(\clog_{\prl}\circ s_{\Z\to\Z'}\circ\cexp_\prl)*\bOm_Z.
\end{equation}


\subsection{Attractor invariants}
\label{sec:attr}
From now on we assume that $\Ga=\bZ^r$ is equipped with a skew-symmetric form $\ang{-,-}$ such that
\begin{equation}
\ang{\ga,\ga'}=0\imp[\cR_\ga,\cR_{\ga'}]=0.
\end{equation}
For any $\ga\in\Ga$, we call the stability parameter $\te=\ang{-,\ga}$ the \idef{self-stability} \cite{bridgeland_scattering} and we call its generic perturbation the \idef{attractor stability} \cite{mozgovoy_attractor}.
We define \idef{attractor invariants} (or initial data), for any $\ga\in S$,
\begin{equation}
\bOm_*(\ga)=\bOm_{\te'}(\ga),\qquad
\hOm_*(\ga)=\hOm_{\te'}(\ga),
\end{equation}
where $\te'$ is a generic perturbation of $\te=\ang{-,\ga}$ (inducing a weak stability from Example \ref{ex:stab}).
Note that \eqref{eq:from A to bar A} implies that
$1$-collections $\bOm_*$ and $\hOm_*$ satisfy
\begin{equation}
\bOm_*=\clog_\prl*\hOm_*.
\end{equation}

\begin{remark}
Let $Q$ be an acyclic quiver and $A_\Z$ be defined as in Example \ref{ex:quiver}.
Then $A_*(\ga)=0$ unless $\ga=ne_i$ for $n\ge1$, $i\in Q_0$,
in which case $A_*(ne_i)=\frac{(-y)^{n^2}}{P(\GL_n;y)}x_i^n$
\cite{mozgovoy_attractor}.
It was observed in \cite{beaujard_vafa,mozgovoy_attractor}
that also in many other situations the attractor invariants have a particularly simple form.
Moreover, we can express 1-collection $A$ in terms of $1$-collections $\bOm_*$ and $\hOm_*$.
Therefore, for any central charge $\Z$, we can express 1-collection $\bOm_\Z$ in terms of~$\bOm_*$.
Several formulas of this type, under the name flow tree formulas,
were conjectured in
\cite{alexandrov_attractor,alexandrov_s}
based on the earlier works 
\cite{denef_splita,denef_split,manschot_walla}.
Later we will prove one of them in the form presented in~\cite{mozgovoy_attractor}.
\end{remark}

\begin{remark}[Alternative definition]
\label{alt-attr}
Let $\te=\ang{-,\ga}$.
Consider a decomposition (we consider below completed direct sums)
\begin{gather}
\fg=\fg^\te_+\oplus\fg^\te_0\oplus\fg^\te_-,
\notag\\
\fg^\te_+=\bop_{\te(\ga')>0}\fg_{\ga'},\qquad
\fg^\te_0=\bop_{\te(\ga')=0}\fg_{\ga'},\qquad
\fg^\te_-=\bop_{\te(\ga')<0}\fg_{\ga'}.
\end{gather}
We have subgroups $G^\te_\pm=\exp(\fg^\te_\pm)\sbs G$, $G^\te_0=\exp(\fg^\te_0)\sbs G$ and a bijection $G_+^\te\xx G_0^\te\xx G_-^\te\to G$ induced by multiplication.
Therefore we have a projection $\pi_\te:G\to G^\te_0$ (which is not a group homomorphism in general).
Collection $\hOm_\te=s_\te\inv*\hOm$ (restricted to $\te^\perp\cap S$) corresponds to the series $\what\hOm_\te=\pi_\te(\what\hOm)\in G^\te_0$.
Further, let us decompose
\begin{equation}
\fg^\te_0=\fg^\prl_0\oplus\fg^\perp_0,
\end{equation}
where $\fg^\prl_0=\bop_{\ga'\prl\ga}\fg_\ga$
and $\fg^\perp_0$ is the sum of the rest of the summands $\fg_\ga$ in $\fg^\te_0$.
Then $\fg^\perp_0\sbs\fg^\te_0$ is an ideal.
Indeed, assume that $\fg_{\ga'}\sbs\fg^\te_0$, $\fg_{\ga''}\sbs\fg^\perp_0$ and
$\fg_{\ga'+\ga''}\not\sbs\fg^\perp_0$.
Then $(\ga'+\ga'')\prl\ga$,
hence
$\ang{\ga',\ga''}=\ang{\ga',\ga'+\ga''}=0$ as $\ang{\ga',\ga}=0$.
Therefore $[\fg_{\ga'},\fg_{\ga''}]=0$ and we always have $[\fg_{\ga'},\fg_{\ga''}]\sbs\fg^\perp_0$.
Therefore we have a group homomorphism $\Psi:G^\te_0\to G^\prl_0=\exp(\fg^\prl_0)$
with the kernel $G^\perp_0=\exp(\fg^\perp_0)$.
Invariant $\hOm_*(\ga)$ can be identified with the degree \ga component of $\Psi(\what\hOm_\te)\in G^\prl_0$ by \cite[Theorem 1.21]{gross_canonicala} (see also \cite{mozgovoy_attractor}).
Similarly, attractor invariant $\bOm_*(\ga)$ can be identified with the degree \ga component of $\log\Psi(\what\hOm_\te)\in\fg^\prl_0$.
\end{remark}

Note that in general
\begin{equation}
\hOm_*(\ga)\ne\hOm_\te(\ga),\qquad
\te=\ang{-,\ga}.
\end{equation}
But we will see below that
$\bOm_*(\ga)=\bOm_\te(\ga)$, where $\bOm_\te$ 
was defined in \eqref{DT1} (note that $\te=\ang{-,\ga}$ is not generic or weakly-generic in general).

\begin{theorem}\label{th:bom*1}
For any $\ga\in S$, we have
\begin{equation}
\bOm_*(\ga)=\bOm_\te(\ga),
\end{equation}
where $\te=\ang{-,\ga}$.
\end{theorem}
\begin{proof}
Consider collection $\hOm_\te=s_\te\inv*\hOm$ (restricted to $\te^\perp\cap S$)
and the corresponding series
$\what\hOm_\te\in G^\te_0$.
Using notation from Remark \ref{alt-attr}, 
consider a commutative diagram
\begin{ctikzcd}
\fg_0^\te\rar["\pi"]\dar["\exp"']&\fg_0^\prl\dar["\exp"]\\
G_0^\te\rar["\Psi"]&G_0^\prl
\end{ctikzcd}
Note that collection $\clog_\te*\hOm_\te$ does not correspond to the series $\log(\what\hOm_\te)\in\fg_0^\te$ as the Lie algebra $\fg^\te_0$ is not abelian in general.
But because $\fg_0^\perp\sbs\fg^\te_0$ is an ideal and $\fg_0^\prl=\fg^\te_0/\fg^\perp_0$ is abelian, we conclude that
$\bOm_\te(\ga)=(\clog_\te*\hOm_\te)(\ga)$ is equal to the degree $\ga$ component of $\pi\log(\what\hOm_\te)$
which is the degree $\ga$ component of $\log(\Psi(\what\hOm_\te))$.
We have seen in Remark \ref{alt-attr} that the latter coincides with $\bOm_*(\ga)$.
\end{proof}

\subsection{Attractor tree formulas}
Recall that in Theorem \ref{th:log-s}, we
proved that collection
$g_\Z=\clog_\Z\circ s_\Z\inv$ is given by
\begin{equation}
g_{\Z}(\al)=\begin{cases}
\frac{(-1)^{n-1}}{n_0+1}& \al_{\le k}\ge_\Z\al_{>k}
\ \forall\, 1\le k<n,\\
0&\text{otherwise,}
\end{cases}
\qquad \al\in S^n,
\end{equation}
where $n_0$ is the number of $1\le k<n$ such that $\al_{\le k}\sim_\Z \al_{>k}$.
Motivated by this result, we define the \idef{attractor collection} $g_*$ as
\begin{equation}
g_{*}(\al)=\begin{cases}
\frac{(-1)^{n-1}}{n_0+1}& 
\ang{\al_{\le k},\al_{>k}}\ge0
\ \forall\, 1\le k<n,\\
0&\text{otherwise,}
\end{cases}
\qquad \al\in S^n,
\end{equation}
where $n_0$ is the number of $1\le k<n$ such that $\ang{\al_{\le k},\al_{>k}}=0$.

\begin{theorem}
\label{th:bom*}
We have
\begin{equation}
\bOm_*=g_**\hOm.
\end{equation}
\end{theorem}
\begin{proof}
For any $\ga\in S$, let $\te=\ang{-,\ga}$.
We conclude from Theorem \ref{th:bom*1} that
$$\bOm_*(\ga)=\bOm_\te(\ga)
=(g_\te*A)(\ga)
=\sum_{\ov{\al\in S^n}{\nn\al=\ga}}
g_\te(\al)\prod_{i=1}^n A(\al_i).$$
Note that if $\nn\al=\ga$, then
$$\ang{\al_{\le k},\al_{>k}}
=\ang{\al_{\le k},\ga}=\te(\al_{\le k}).$$
Therefore $\ang{\al_{\le k},\al_{>k}}\ge0$ if and only if $\te(\al_{\le k})\ge0$ if and only if $\mu_\te(\al_{\le k})\ge\mu_\te(\al_{>k})$ (and similarly for equalities).
This implies that $g_\te(\al)=g_*(\al)$, whenever $\nn\al=\ga$.
Therefore $\bOm_*(\ga)=(g_**A)(\ga)$.
\end{proof}

This result implies that
$A=g_*\inv*\bOm_*=g_\Z\inv*\bOm_Z$,
for any central charge \Z.
Therefore
\begin{equation}\label{eq:* to Z}
\bOm_Z=(g_\Z\circ g_*\inv)*\bOm_*.
\end{equation}

\begin{theorem}[Attractor tree formula]
\label{th:attr tree1}
For any central charge \Z, we have $\bOm_\Z=F_\Z*\bOm_*$, or equivalently
\begin{equation}
\bOm_\Z(\ga)=\sum_{\ov{\al\in S^n}{\al_1+\dots+\al_n=\ga}}F_\Z(\al)\cdot
\prod_{i=1}^n\bOm_*(\al_i),
\end{equation}
where $F_\Z:S^*\to\bQ$ is a collection
such that $F_\Z(\al)=1$ for $\al\in S$ and
for $\al\in S^n,\ n\ge2$,
\begin{equation}
F_\Z(\al)=\sum_T\sbr{
(-1)^{\n{V(T)}-1}(g_\Z(\al|_{\ch v_0})-g_*(\al|_{\ch v_0}))
\prod_{v\in V(T)\ms\set {v_0}}g_*(\al|_{\ch v})
}
\end{equation}
is a sum over plane rooted trees $T$ with leaves $1,\dots,n$ and with internal vertices having at least two children.
Here $v_0$ is the root and $\al|_{\ch v}\in S^{\ch v}$ is a tuple defined in \eqref{tuple on children}.
\end{theorem}
\begin{proof}
Note that $g_*(\al)=1$ for all $\al\in S$.
Therefore $\one-g_*$ is supported in cardinality $\ge2$ and $g_*\inv=\bT(\one-g_*)$ by Remark \ref{rm:inverse}.
We can rewrite the above formula 
for $F_\Z$ in the form
$$F_\Z
=\one+(g_\Z-g_*)\circ \bT(\one-g_*)
=\one+(g_\Z-g_*)\circ g_*\inv
=g_\Z\circ g_*\inv
.$$
We have seen that
$\bOm_\Z=(g_\Z\circ g_*\inv)*\bOm_*=F_\Z*\bOm_*$.
\end{proof}

We are going to prove now a slightly more general result that uses geometric collections from \S\ref{sec:geom col}.
Recall that for any central charge $\Z$ and $t\in\bQ$,
there is a collection \eqref{eq:g-te-t1}
\begin{equation}
g_{\Z,t}=g_\Z\circ\si_t=\clog_\Z\circ s_\Z\inv
\circ \si_t.
\end{equation}
which is given, for any $\al\in S^n$, by
\begin{equation}
g_{\Z,t}(\al)
=t^{n_-}(t-1)^{n_+}\frac{t^{n_0+1}-(t-1)^{n_0+1}}{n_0+1}
\end{equation}
where $n_+$ (respectively $n_-$ and $n_0$) is the number of $1\le k<n$ such that $\al_{\le k}>_\Z\al_{>k}$ 
(respectively $<_\Z$ and $\sim_\Z$).
Let us define collection $g_{*,t}$
which is given, for any $\al\in S^n$, by
\begin{equation}
g_{*,t}(\al)
=t^{n_-}(t-1)^{n_+}\frac{t^{n_0+1}-(t-1)^{n_0+1}}{n_0+1}
\end{equation}
where $n_+$ (respectively $n_-$ and $n_0$) is the number of $1\le k<n$ such that $\ang{\al_{\le k},\al_{>k}}>0$ 
(respectively $<0$ and $=0$).

\begin{theorem}[Attractor tree formula. General version]
\label{th:attr tree2}
For any central charge \Z, and any $t\in\bQ$,
we have $\bOm_\Z=F_{\Z,t}*\bOm_*$, or equivalently
\begin{equation}
\bOm_\Z(\ga)=\sum_{\ov{\al\in S^n}{\al_1+\dots+\al_n=\ga}}F_{\Z,t}(\al)\cdot
\prod_{i=1}^n\bOm_*(\al_i),
\end{equation}
where $F_{\Z,t}:S^*\to\bQ$ is a collection
such that $F_{\Z,t}(\al)=1$ for $\al\in S$ and
for $\al\in S^n,\ n\ge2$,
\begin{equation}
F_{\Z,t}(\al)=\sum_T\sbr{
(-1)^{\n{V(T)}-1}(g_{\Z,t}(\al|_{\ch v_0})-g_{*,t}(\al|_{\ch v_0}))
\prod_{v\in V(T)\ms\set {v_0}}g_{*,t}(\al|_{\ch v})
}
\end{equation}
is a sum over plane rooted trees $T$ with leaves $1,\dots,n$ and with internal vertices having at least two children.
Here $v_0$ is the root and $\al|_{\ch v}\in S^{\ch v}$ is a tuple defined in \eqref{tuple on children}.
\end{theorem}
\begin{proof}
As in the previous theorem, we can rewrite the above formula in the form
$F_{\Z,t}=g_{\Z,t}\circ g_{*,t}\inv$.
As in Theorem \ref{th:bom*}, we have
$$\bOm_*(\ga)=(g_\te*\hOm)(\ga)
=(g_{\te,t}*(\si_t\inv*\hOm))(\ga)
$$
for $\te=\ang{-,\ga}$.
Using again the arguments of Theorem \ref{th:bom*}, we obtain
$$\bOm_*(\ga)
=(g_{*,t}*(\si_t\inv*\hOm))(\ga).$$
Therefore
$$\si_t\inv*\hOm=g_{*,t}\inv*\bOm_*
=g_{\Z,t}\inv*\bOm_\Z
$$
and 
$\bOm_\Z
=(g_{\Z,t}\circ g_{*,t}\inv)*\bOm_*
=F_{\Z,t}*\bOm_*$.
\end{proof}

\begin{remark}
Note that $g_{\Z,t}=g_\Z\circ\si_t$,
but $g_{*,t}\ne g_*\circ\si_t$ in general.
However, in the course of the proof of the above theorem we proved that $g_{*,t}\circ\si_t\inv$ acts in the same way as $g_*$ on $1$-collections.
Therefore $g_{*,t}$ acts in the same way as $g_*\circ\si_t$.
Similarly, collections $F_{\Z,t}$ from the above theorem act in the same way on $1$-collections, for different $t$.
But these collections can depend on $t$
and $F_{\Z,t}\ne F_{\Z,0}=F_\Z=g_\Z\circ g_*\inv$ in general.
\end{remark}

\begin{remark}
The formulas for $g_{\Z,t}(\al)$ for special values $t=0,\,1,\,\oh$ were determined in Remark~\ref{rm:t=1/2}.
Similarly, we can write the corresponding formulas for $g_{*,t}(\al)$ as follows.
Let $n_+$ (respectively $n_-$ and $n_0$) be the number of $1\le k<n$ such that $\ang{\al_{\le k},\al_{>k}}>0$ 
(respectively $<0$ and $=0$). Then
%
\begin{enumerate}
\item 
For $t=0$, we have
$$g_{*,0}(\al)=
\begin{cases}
\frac{(-1)^{n-1}}{n_0+1}&\al_{\le k}\ge_\te\al_{>k}\
\forall 1\le k<n,\\
0&\text{otherwise.}
\end{cases}
$$
Therefore $g_{*,0}=g_*$ and Theorem \ref{th:attr tree1} is a special case of Theorem \ref{th:attr tree2}.
\item
For $t=1$, we have
$$g_{*,1}(\al)=
\begin{cases}
\frac{1}{n_0+1}&\al_{\le k}\le_\te\al_{>k}\
\forall 1\le k<n,\\
0&\text{otherwise.}
\end{cases}
$$
\item
For $t=\oh$, we have
\begin{equation*}
g_{*,\oh}(\al)=\frac{(-1)^{n_+}}{2^n}\cdot
\frac{1+(-1)^{n_0}}{n_0+1}.
\end{equation*}
Note that it is zero for odd $n_0$.
The corresponding attractor tree formula was conjectured in \cite{alexandrov_s} (\cf \cite{mozgovoy_attractor}).
\end{enumerate}
\end{remark}

\section{Lie collections}
\label{app1}
\subsection{Lie collections and plethysm}
Let $\bQ S=\bop_{\ga\in S}\bQ$ and $\Ass(S)=\Ass(\bQ S)$ be the free associative (non-unital) algebra generated by $S$.
Note that $\Ass(S)$ has a basis $S^*=\sqcup_{n\ge1}S^n$ and has a filtration induced by the length of words in $S^*$. 
Let $\what\Ass(S)$ be its completion so that
\begin{equation}
\what\Ass(S)\iso\prod_{\al\in S^*}\bQ\iso\Map(S^*,\bQ)
\end{equation}
as a vector space.
A collection $F:S^*\to\bQ$ can be identified with an element $F\in\what\Ass(S)$.
Let $\cR$ be an $S$-graded associative algebra and $H:S\to\cR$ be a $1$-collection such that $H(\ga)\in\cR_\ga$ for $\ga\in S$.
Then $H$ induces an algebra homomorphism $H:\Ass(S)\to\cR$ and $H:\what\Ass(S)\to\what\cR$.
The $1$-collection
$F*H:S\to\cR$ defined in \eqref{action2} can be identified with
$\sum_{\ga\in S}(F*H)(\ga)\in\what\cR$ which is equal to $H(F)\in \what\cR$.
\medskip

Similarly, let $\Lie(S)=\Lie(\bQ S)$ be the free Lie algebra generated by $S$.
We can represent it as a Lie subalgebra $\Lie(S)\sbs\Ass(S)$.
Let $\what\Lie(S)\sbs\what\Ass(S)$ be the completion of $\Lie(S)$.
We say that $F\in\Map(S^*,\bQ)\iso\what\Ass(S)$ is a 
\idef{Lie collection} if $F\in\what\Lie(S)$.
Let $\fg$ be an $S$-graded Lie algebra and $H:S\to\fg$ be a $1$-collection such that $H(\ga)\in\fg_\ga$ for $\ga\in S$.
Then $H$ induces a Lie algebra homomorphism $H:\Lie(S)\to\fg$ as well as $H:\what\Lie(S)\to\what\fg$.
We define a $1$-collection $F*H:S\to\fg$ such that $(F*H)(\ga)$ is equal to the degree $\ga$ component of $H(F)\in\what\fg$.
Similarly, given $F,\,G\in\what\Lie(S)$, we define their plethysm $F\circ G\in\what\Lie(S)$ as follows.
The Lie algebra $\Lie(S)$ (as well as $\what\Lie(S)$) is $S$-graded, where $\ga\in S\sbs\Lie(S)$ has degree $\ga$.
We can interpret $G\in\what\Lie(S)$ as a map  $G':S\to\Lie(S)$, where $G'(\ga)$ is the degree $\ga$ component of $G$.
Then $G'$ induces a Lie algebra homomorphism $G':\Lie(S)\to\Lie(S)$ as well as $G':\what\Lie(S)\to\what\Lie(S)$ and we define $F\circ G$ to be the element $G'(F)\in\what\Lie(S)$.
Note that this plethysm is equal to the plethysm of $F$ and $G$ considered as \qq{associative} collections which was defined 
in \eqref{eq:plethysm}.
We will prove the following result in \S\ref{sec:col operads}.

\begin{theorem}
\label{th:Lie plethysm}
Given two Lie collections $F,G\in\what\Lie(S)$ and a $1$-collection $H:S\to\fg$ as above, 
we have $(F\circ G)*H=F*(G*H)$.
\end{theorem}

Note that a Lie collection $F\in\what\Lie(S)$ is determined by the coefficients of $F$ on a basis of $\Lie(S)$.
Such basis can be constructed using Lyndon words in $S$.
Given an order on $S$, a word $w\in S^*$ is called a \idef{Lyndon word} if for any factorization $w=uv$ (with $u,\,v\in S^*$), we have $w<v$ in lexicographical order on $S^*$. 
Given a Lyndon word $w$ (of length $\ge2$), consider its \idef{standard factorization} $w=uv$, where $u,v$ are Lyndon words and $v$ is the longest possible (note that $u<v$).
Continuing this process we obtain a (plane) binary tree with leaves labeled by elements of $S$.
Taking the Lie brackets at every internal vertex of the tree we obtain the required element $b(w)\in\Lie(S)$.
These elements form a \idef{Lyndon-Shirshov basis} of $\Lie(S)$.
\medskip


Let us assume now that we are in the setting of \S\ref{sec:DT} and let $\Z,\Z'$ be two central charges.
Then we have $1$-collections $\bOm_\Z$ and $\bOm_{\Z'}$ in $\cR$ which we can consider as $1$-collections in the Lie algebra $\fg=\bop_{\ga\in S}\cR_\ga\sbs\cR$.
They are related by \eqref{bom Z to Z'}
\begin{equation}
\bOm_{\Z'}
=(g_{\Z'}\circ g_Z\inv)*\bOm_\Z
=(\clog_{\Z'}\circ s_{\Z\to\Z'}\circ\cexp_\Z)*\bOm_Z,
\end{equation}
A priori the above collection is contained in $\what\Ass(S)$, but we have the following result that follows from \cite{joyce_configurations}.

\begin{theorem}
Collection $\clog_{\Z'}\circ s_{\Z\to\Z'}\circ\cexp_\Z$
is contained in $\what\Lie(S)$.
\end{theorem}

It would be interesting to find an explicit formula for this collection in some basis of $\Lie(S)$.
An indication of how this can be done is given by the flow tree formula that we will discuss next.

\subsection{Flow tree formula}
Let $\cT_\bin(n)\sbs\cT(n)$ be the set of plane binary trees with leaves $1,\dots,n$
(note that $\cT_\bin(1)$ contains just the unit tree
\S\ref{sec:trees}).
For any $\al\in S^n$ and $T\in\cT_\bin(n)$,
define the element $\pi(T,\al)\in\Lie(S)$ 
by taking Lie brackets at the internal vertices of $T$ and taking $\al_i\in S$ at the leaves $i\in L(T)$.
The set of such elements generates $\Lie(S)$, but they are linearly dependent.

Let $\fg$ be an $S$-graded Lie algebra and assume, similarly to \S\ref{sec:attr}, that we have a skew-symmetric form on $\Ga=\bZ^r\sps S$ such that
\begin{equation}
\ang{\ga,\ga'}=0\imp[\fg_\ga,\fg_{\ga'}]=0.
\end{equation}
Let
\begin{equation}
F=\sum_{T,\al}f(T,\al)\cdot\pi(T,\al)\in\what\Lie(S)
\end{equation}
be a Lie collection with $f(T,\al)\in\bQ$ and let $H:S\to\fg$ be a $1$-collection (with $H(\ga)\in\fg_\ga$).
Then $F*H$ depends just on $f(T,\al)$ such that 
$\ang{\al_{v_1},\al_{v_2}}\ne0$, for any vertex $v\in V(T)$ with $\ch(v)=(v_1,v_2)$ (see \eqref{tuple on children} for the definition of $\al_v$).
Moreover, performing permutations of the children, we can assume that $f(T,\al)\ne0$ only if
$\ang{\al_{v_1},\al_{v_2}}>0$ for any vertex $v\in V(T)$
with $\ch(v)=(v_1,v_2)$.
Note that the corresponding elements $\pi(T,\al)\in\Lie(S)$ 
are linearly dependent. 
\medskip

Now assume that we are in the setting of \S\ref{sec:attr}
and consider a $1$-collection $\bOm_*$ of attractor invariants
in the Lie algebra $\fg=\bop_{\ga\in S}\cR_\ga\sbs\cR$.
For any central charge $\Z$, we have a $1$-collection
$\bOm_{\Z}$ in $\fg$ which satisfies \eqref{eq:* to Z}
\begin{equation}
\bOm_Z=(g_\Z\circ g_*\inv)*\bOm_*.
\end{equation}
We expect that $g_\Z\circ g_*\inv$ (or its modification that acts in the same way on $1$-collections) is contained in $\what\Lie(S)$.
The following explicit formula for such collection is a generalization of the flow tree formula conjectured in \cite{alexandrov_attractor}.
The proof of the original flow tree formula was recently obtained in \cite{bousseau_flow} and the proof of the conjecture below should go through the same lines.

\begin{conjecture}[Flow tree formula]
For any generic central charge $\Z=-\te+\bi\rho$,
we have $\bOm_\Z=F_\Z*\bOm_*$, 
\begin{equation}
F_\Z=\sum_{\ov{[\al]\in S^n/\fS_n}{n\ge1}}
\frac1{\n{\Aut\al}}
\sum_{\si\in\fS_n}\sum_{T\in\cT_\bin(n)}
\eps_\Z(T,\al,\si)\cdot\pi(T,\al^\si)\in\what\Lie(S),
\end{equation}
where $\al^\si=(\al_{\si i})_i\in S^n$,
$\Aut\al\sbs\fS_n$ is the stabilizer of $\al\in S^n$ under the action of $\fS_n$
and the coefficients $\eps_\Z(T,\al,\si)\in\set{0,1}$
are defined as follows.
%
Let $\nn\al=\sum_{i=1}^n\al_i$ and $\te'=\te-\mu_\Z(\nn\al)\rho$ so that $\te'(\nn\al)=0$.
Consider the skew-symmetric form on $\bZ^n$ defined by
$\ang{e_i,e_j}=\ang{\al_{\si i},\al_{\si j}}$
and a linear map $\te^\si:\bZ^n\to\bR$ that is a generic perturbation of the linear map $e_i\mto\te'(\al_{\si i})$
such that $\te^\si(\sum_i e_i)=0$
(we assume that perturbations for different $\si\in\fS_n$ are compatible).
Define linear maps $\te_v:\bZ^n\to\bR$ for $v\in V(T)$ as follows

\begin{enumerate}
\item For any $v\in V(T)$, let $e_v=\sum_{i\in L(v)}e_i$ 
and $e'_v=e_{v_1}$, $e''_v=e_{v_2}$ with $(v_1,v_2)=\ch(v)$.
\item
Add a \qq{parent} $w$ of the root $v_0\in T$
and define $\te_{w}=\te^\si$ so that $\te_{w}(e_{v_0})=0$.
\item Assume that $\te_p$ for the parent $p$ of $v\in V(T)$ is known and $\te_p(e_v)=0$. Then define
\begin{equation}
\te_v(\ga)=\te_p(\ga)-\frac{\te_p(e'_v)}{\ang{e'_v,e_v}}
\ang{\ga,e_v},\qquad\ga\in\bZ^n,
\end{equation}
where $\ang{e'_v,e_v}=\ang{e'_v,e''_v}$ is assumed to be nonzero \eqref{flow tree coeff}.
Note that $\te_v(e'_v)=\te_v(e''_v)=0$.
\end{enumerate}
Finally, we define
\begin{equation}\label{flow tree coeff}
\eps_\Z(T,\al,\si)=
\begin{cases}
1&\ang{e'_v,e''_v}>0,\,\te_p(e'_v)<\te_p(e''_v)\quad
\forall\, v\in V(T),\, p=p(v),\\
0&\text{otherwise}.
\end{cases}
\end{equation}
\end{conjecture}

\begin{figure}[ht]
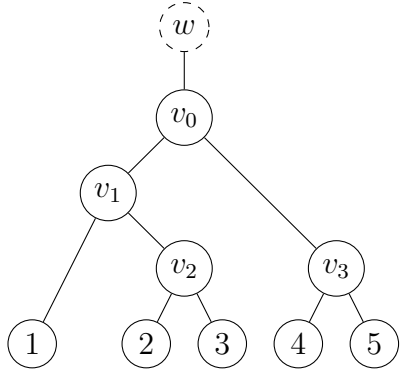

\begin{ctikz}
\node[crc1,dashed](w)at(0,1.2){$w$};
\node[crc1](0)at(0,0){$v_0$};
\node[crc1](1)at(-1,-1){$v_1$};
\node[crc1](3)at(0,-2){$v_2$};
\node[crc1](4)at(2,-2){$v_3$};
\node[crc1](l1)at(-2,-3){$1$};
\node[crc1](l2)at(-.5,-3){$2$};
\node[crc1](l3)at(.5,-3){$3$};
\node[crc1](l4)at(1.5,-3){$4$};
\node[crc1](l5)at(2.5,-3){$5$};
\draw (w)--(0)
(0)--(1)--(l1) (1)--(3)
(3)--(l2)(3)--(l3)
(0)--(4)--(l5) (4)--(l4);
\end{ctikz}
\caption{A binary tree $T\in\cT_{\bin}(5)$ with the root $v_0$
and its \qq{parent} $w$.}
\end{figure}

\begin{remark}
Note that the above flow tree formula is slightly different from the formula conjectured in
\cite{alexandrov_attractor}.
In \cite{alexandrov_attractor} one requires just that $\ang{e'_v,e''_v}$ and $\te_p(e'_v)-\te_p(e''_v)$ have opposite signs.
This leads to more nonzero contributions and one needs to divide the resulting coefficient by $2^{n-1}$ which corresponds to permutations at the internal vertices of $T$.
Moreover, in \cite{alexandrov_attractor} one considers a generic perturbation of the skew-symmetric form, while in the above formula we consider a generic perturbation of the stability parameter.
\end{remark}

\begin{remark}
Using notation from Example \ref{ex:quiver}, let
\begin{equation}
\bOm_\Z(\ga)=\frac{\bar\Om_\Z(\ga)}{y\inv-y}x^\ga,\qquad
\bOm_*(\ga)=\frac{\bar\Om_*(\ga)}{y\inv-y}x^\ga
\end{equation}
for
$\bar\Om_\Z(\ga),\,\bar\Om_*(\ga)\in\bQ(y)$.
Then the above conjecture can be written in the form
\begin{equation}
\bar\Om_\Z(\ga)=
\sum_{\ov{[\al]\in S^n/\fS_n}{\nn\al=\ga}}
\frac1{\n{\Aut\al}}
\sum_{\si\in\fS_n}\sum_{T\in\cT_\bin(n)}
\eps_\Z(T,\al,\si)\kappa(T,\al^\si)\prod_{i=1}^n\bar\Om_*(\al_i),
\end{equation}
\begin{equation}
\kappa(T,\al)=
\prod_{\ov{v\in V(T)}{\ch(v)=(v_1,v_2)}}
\kappa(\ang{\al_{v_1},\al_{v_2}}),\qquad
\kappa(m)=\frac{(-y)^m-(-y)^{-m}}{y\inv-y}.
\end{equation}
This is the form of the flow tree formula conjectured in \cite{alexandrov_attractor}.
Note that it can be specialized to $y=1$.
\end{remark}

\subsection{Collections in algebras over operads}
\label{sec:col operads}
To prove Theorem \ref{th:Lie plethysm} it is convenient to generalize it first to arbitrary operads.
Let $\cV$ be a symmetric monoidal category and $\cP$ be a symmetric operad in \cV.
Let $\cV^S$ denote the category of $S$-graded objects in \cV, with an induced symmetric monoidal category structure.
The operad $\cP$ induces an operad structure in $\cV^\Ga$,
hence a monad $\cP:\cV^\Ga\to \cV^\Ga$ and 
$\cP:\cV^S\to \cV^S$ (recall that $\Ga$ is a lattice and $S\sbs\Ga$).
Let $A\in\cV^S$ be a \cP-algebra.
Then, motivated by the previous discussion, we define a collection in $A$ to be a morphism in $\cV^S$
\begin{equation}
H:S\ts\one\to A,\qquad 
S\ts\one=\bop_{\ga\in S}\one\in\cV^S,
\end{equation}
where $\one\in\cV$ is the identity object (for $\cV=\Vect_\bQ$, we have $\one=\bQ$).

\begin{remark}
The functor
$$S\ts-:\cV\to\cV^S,\qquad V\mto S\ts V=\bop_{\ga\in S}V$$
has a right adjoint
$$\Hom(S,-):\cV^S\to\cV,\qquad W\mto\prod_{\ga\in S}W_\ga.$$
Therefore a $1$-collection $H:S\ts\one\to A$ can be identified with $H:\one\to\prod_{\ga\in S}A_\ga=:\what A$.
For $\cV=\Vect_\bQ$ and $\one=\bQ$, we obtain $H\in\what A$.
\end{remark}

Consider the free $\cP$-algebra $\cP W\in\cV^S$, where $W=S\ts\one$.
Motivated by the previous discussion, we define a collection to be a morphism $F:W\to\cP W$ in $\cV^S$
or equivalently
\begin{equation}
F:\one\to\Hom(S,\cP W)=\prod_{\ga\in S}\cP(W)_\ga=:\what\cP(W).
\end{equation}
Note that for $\cV=\Vect_\bQ$, we have $\one=\bQ$, $S\ts\one=\bQ S$ and $F\in\what\cP(\bQ S)$.

We define plethysm 
$F\circ H:W\to A$ 
between
$F:W\to\cP W$ and $H:W\to A$ to be the composition
\begin{equation}
W\xto F\cP W\xto{\cP H}\cP A\xto\mu A, 
\end{equation}
where $\mu:\cP A\to A$ is the multiplication structure of the \cP-algebra $A$.
One can see that in the case $\cV=\Vect_\bQ$ and $\cP=\Ass$ or $\cP=\Lie$, plethysm $F\circ H\in\what A$ corresponds to the $1$-collection $F*H:S\to A$ we defined earlier.

In particular, for two collections $F,\, G:W\to\cP W$, we have plethysm $F\circ G:W\to\cP W$.
Again, one can see that in the case $\cV=\Vect_\bQ$ and $\cP=\Ass$, plethysm $F\circ G\in\what\cP(W)$ corresponds to 
the plethysm $F\circ G\in\what\Ass(S)$ we defined earlier.
In the case $\cV=\Vect_\bQ$ and $\cP=\Lie$, we have collections $F,G\in\what\Lie(S)$ and their plethysm $F\circ G\in\what\Lie(S)$.
Now the statement of the theorem follows from the following result.

\begin{lemma}
Given collections $F,\, G:W\to\cP W$ and $H:W\to A$,
we have
$$(F\circ G)\circ H=F\circ(G\circ H).$$
\end{lemma}
\begin{proof}
The statement follows from the commutativity of the diagram
\begin{ctikzcd}
W\rar["F"]&\cP W\rar["\cP G"]&
\cP^2W\rar["\cP^2H"]\dar["\mu"]&
\cP^2A\dar["\mu"]\rar["\cP\mu"]&
\cP A\dar["\mu"]\\
&&\cP W\rar["\cP H"]&\cP A\rar["\mu"]&A
\end{ctikzcd}
\end{proof}

\appendix
\section{Relation to operadic categories}
\label{app2}
The goal of this appendix is to embed our wall-crossing setting developed in the previous sections into the framework of operadic categories.
We will discuss only the most necessary definitions and constructions without delving into finer details.
An interested reader can consult \cite{batanin_operadic,lack_operadic} for more details on operadic categories. 

\subsection{From a semigroup to an operadic category}
\label{from to opercat}
Let $S$ be a commutative semigroup.
As in \S\ref{sec:col plethysm}, we assume that every element in $S$ can be represented as a finite sum of other elements in finitely many ways.
Recall that for any map $\vi:I\to J$ between finite sets,
we define
\begin{gather}
\vi_*:S^I\to S^J,\qquad (\vi_*\al)_j=\sum_{i\in \vi\inv(j)}\al_i,\qquad \al\in S^I,\ j\in J,\\
\vi^*:S^J\to S^I,\qquad (\vi^*\al)_i=\al_{\vi(i)},\qquad \al\in S^J,\ i\in I,
\end{gather}

Let $\De$ be the category of finite non-empty ordered sets and order-preserving maps between them, called the \idef{simplex category}.
Consider the functor
\begin{equation}
W:\De\to\Set
\end{equation}
defined by $W(I)=S^I$ and $W(\vi)=\vi_*:W(I)\to W(J)$, for a morphism $\vi:I\to J$ in \De.
Let $\cC=\Ga(W)$ be the \idef{Grothendieck construction} for the functor $W$.
This is a category whose objects are pairs $(I,\al)$, where $I\in\De$ and $\al\in W(I)$,
and morphisms $\vi:(I,\al)\to(J,\be)$ are morphisms $\vi:I\to J$ in $\De$ such that $\be=\vi_*(\al)$.
Note that the set of isomorphism classes of objects in $\cC$ can be identified with $S^*=\bigsqcup_{n\ge1}S^n$.
The forgetful functor
\begin{equation}
\n-:\cC\to\fSet,\qquad (I,\al)\mto I,
\end{equation}
where \fSet denotes the category of finite sets,
will be called the \idef{cardinality functor}.
For a morphism $\vi:(I,\al)\to(J,\be)$ as above we will sometimes write $\n \vi:I\to J$ to distinguish it from the morphism $\vi$.

Given an object $b=(J,\be)$ in \cC, 
consider the comma category $\cC/b$ \cite[\S II.6]{maclane_categories} whose objects are morphisms $\vi:a\to b$ in \cC and morphisms from $\vi:a\to b$ to $\vi':a'\to b$ are morphisms $\psi:a\to a'$ in $\cC$ such that $\vi'\psi=\vi$.
We define a functor
\begin{equation}
F:\cC/b\to\cC^{\n{b}}=\cC^J,
\end{equation}
called the \idef{fiber functor}, as follows.
For any morphism $\vi:a=(I,\al)\to(J,\be)=b$ (an object in $\cC/b$) and any $j\in J$,
we define the object
\begin{equation}
F_{j}(\vi)=\vi\inv j=(I_j,\al|_{I_j})\in\cC,
\end{equation}
called the \idef{fiber of $\vi$},
where $\io_j:I_j=\n \vi\inv (j)\emb I$ is an embedding and
\begin{equation}
\al|_{I_j}=\io_j^*(\al)\in W(I_j).
\end{equation}
is a restriction.
The category $\cC$ equipped with the cardinality functor and the fiber functors is an example of an operadic category \cite{batanin_operadic,lack_operadic} which we will recall below.

\subsection{Operadic categories}
One defines an \idef{operadic category} (see \cite{batanin_operadic} and a slight generalization~ \cite{lack_operadic}) 
to be a category $\cC$ equipped with a functor $\n-:\cC\to\fSet$,
called a \idef{cardinality functor}, and functors $F:\cC/b\to\cC^{\n b}$ for all $b\in\cC$,
called \idef{fiber functors},
such that the following conditions are satisfied
\begin{enumerate}
\item
For any $b\in\cC$, the diagram
\begin{ctikzcd}
\cC/b\rar["F"]\dar["\n-"']&\cC^{\n b}\dar["\n-"]\\
\Set\!/\n b\rar["F"]&\Set^{\n b}
\end{ctikzcd}
commutes,
where $F:\Set\!/J\to\Set^J$ is defined by 
$F(\vi:I\to J)=(\vi\inv (j))_{j\in J}$.
For any $\vi:a\to b$ in \cC and $j\in\n b$,
define the \idef{fiber}
$\vi\inv j:=F_j(\vi)\in\cC$.
It satisfies $\n{\vi\inv j}=\n\vi\inv (j)$.
\item
The identity morphism $1_b:b\to b$ is trivial for every $b\in\cC$.
Here a morphism $\vi:a\to b$ in \cC is called (fiberwise) \idef{trivial} if its fibers are trivial. An object $b\in\cC$ is called trivial if $\n b=1$ and
the fiber functor $F:\cC/b\to\cC$ is the forgetful functor $(a\to b)\mto a$.
\item
(Compatibility condition on taking fibers twice \cite{lack_operadic}).
Given morphisms $a\xto\vi b\xto\psi c$ in \cC,
$$\begin{tikzcd}
a\ar[rr,"\phi"]\ar[dr,"\psi\phi"']&& b\ar[dl,"\psi"]\\
&c
\end{tikzcd}$$
consider $\bar\vi:\psi\vi\to\psi$ in $\cC/c$ and
$F_j(\bar\vi):F_j(\psi\vi)\to F_j(\psi)=\psi\inv j$ for $j\in\n c$.
Then
\begin{equation}
F(\vi)=\prod_{j\in\n c}F(F_j(\bar \phi))\in\prod_{j\in\n c}\cC^{\n{\psi\inv j}}=\cC^{\n b}.
\end{equation}
\end{enumerate}

\begin{remark}
\label{rem:trivial fibers}
If $\psi:b\to c$ is trivial, then $\n\psi:\n b\to\n c$ is a bijection.
Using the notation of the third axiom,
consider $\bar\vi:\psi\vi\to\psi$ and
$F_j(\bar\vi):F_j(\psi\phi)\to\psi\inv j$
for $j\in\n c$.
As $\psi\inv j$ is trivial, we obtain
$F(F_j(\bar\vi))=F_j(\psi\phi)$, hence
$F(\vi)=\prod_{j\in\n c}F_j(\psi\vi)=F(\psi\vi)$,
where $\cC^{\n b}$ and $\cC^{\n c}$ are identified using the bijection $\n\psi:\n b\to\n c$.
\end{remark}

Let $\cC_0\sbs\cC$ be the subcategory consisting of all objects in \cC and fiberwise trivial morphisms between them.
We will usually assume that $\cC_0$ is a skeletal category, meaning that every isomorphism class has just one representative.
Given a symmetric monoidal category $\cV$, we define the category $\Col(\cC,\cV)$ of \idef{$\cC$-collections} in $\cV$ to be the category of functors $\Fun(\cC_0^\op,\cV)$.

\begin{example}
\label{ex:symm col}
Let $\cC$ be the category of (non-empty) finite sets, $\n-:\cC\to\fSet$ be the identity functor and fiber functors be defined in the obvious way.
Then $\cC_0=\bS$ is the groupoid consisting of finite sets and bijections between them.
A $\cC$-collection in \cV is the usual symmetric collection $(X(n)\in\cV)_{n\ge1}$, where $X(n)$ is equipped with an action of the symmetric group $\fS_n$.
Let $\cC=\De$ be the category of (non-empty) finite ordered sets and order-preserving maps between them,
$\n-:\cC\to\fSet$ be an embedding and fiber functors be defined in the obvious way.
Then $\cC_0$ has objects parametrized by $n\ge1$ (up to an isomorphism) and only identity morphisms between them.
A $\cC$-collection in \cV is a collection of objects $(X(n)\in\cV)_{n\ge1}$ which is the usual non-symmetric collection in \cV.
\end{example}

\begin{example}
Let $\cC=\Ga(W)$ be the category constructed in \S\ref{from to opercat}
and let $\n-$ and $F$ be the functors constructed there.
Then $\cC$ is an operadic category.
The subcategory $\cC_0\sbs\cC$ has objects parametrized by pairs $(n,\al)$ (up to an isomorphism), where $n\ge1$ and $\al\in S^n$, and only identity morphisms between them.
A \cC-collection in \cV is a collection of objects $(X(\al)\in\cV)_{\al\in S^*}$, where $S^*=\bigsqcup_{n\ge1}S^n$.
\end{example}

\subsection{Plethysm}
From now on we assume that $\cV$ has all necessary colimits.
Assuming that $\cC_0$ has only identity morphisms, we define a plethysm $X\circ Y$ of two collections $X,\,Y\in\Col(\cC,\cV)$ to be a new collection given by 
\begin{equation}
(X\circ Y)(a)=\bop_{\vi:a\to b}X(b)\ts Y(\vi),\qquad 
Y(\vi)=\bts_{j\in\n b}Y(\vi\inv j),
\end{equation}
where the sum runs over all objects $\vi$ in the comma category $a/\cC$.

\begin{example}
Let $\cC=\Ga(W)$ be as in \S\ref{from to opercat}.
A morphism $\vi:a=(I,\al)\to (J,\be)=b$ is an order-preserving map $\pi=\n\vi:I\to J$ such that $\be=\pi_*(\al)$.
For any $j\in J=\n b$, we have $\vi\inv j=(I_j,\al|_{I_j})$, where $I_j=\pi\inv j\sbs I$.
Therefore, for two collections $X,\,Y:S^*\to\cV$, we
have
\begin{equation}
(X\circ Y)(\al)
=\bop_{\pi:I\to J}X(\pi_*\al)\ts\bts_{j\in J}Y(\al|_{\pi\inv j}),\qquad \al\in S^I,
\end{equation}
where the sum runs over all order preserving surjective maps $\pi:I\to J$ (we assume that $Y(\es)=0$).
This formula should be compared to the definition of plethysm in \eqref{eq:plethysm}.
\end{example}

In the case of $\cC_0$ having arbitrary morphisms,
the definition of plethysm is slightly more complicated.
We will propose the definition inspired by \cite[Proposition 9.2]{lack_operadic}, where the case $\cV=\Set$ was considered.
We have seen in Remark \ref{rem:trivial fibers},
that given a morphism $\vi:a\to b$ and a fiberwise trivial morphism
$\psi:b\to c$, we have
$\vi\inv i=(\psi\vi)\inv j$, for $i\in\n b$ and $j=\n\psi (i)\in\n c$.
Given collections $X,\,Y\in\Col(\cC,\cV)$, we 
obtain an isomorphism $f:Y(\vi)\iso Y(\psi\phi)$ and a map
\begin{equation}
X(c)\ts Y(\psi\vi)\xto{X(\psi)\ts f\inv} X(b)\ts Y(\vi).
\end{equation}
Consider the comma category $a/\cC/\cC_0$ whose objects are morphisms $\vi:a\to b$ in \cC and morphisms from $\vi:a\to b$ to $\vi':a\to c$ are morphisms $\psi:b\to c$ in $\cC_0$ such that $\vi'=\psi\vi$.
Then we have a functor
\begin{equation}
(a/\cC/\cC_0)^\op\to\cV,\qquad (\vi:a\to b)\mto X(b)\ts Y(\vi)
\end{equation}
and we define
\begin{equation}
(X\circ Y)(a)=\colim_{(\vi:a\to b)\in a/\cC/\cC_0}X(b)\ts Y(\vi)
\end{equation}
To avoid any difficulties let us assume that a morphism in \cC is fiberwise trivial if and only if it is invertible.
The unit collection $\one:\cC_0^\op\to\cV$ is defined by $\one(a)=\one\in\cV$ if $a\in \cC_0$ is a trivial object (in particular $\n a=1$) and zero otherwise.
All morphisms are sent to identities.
One defines an \idef{operad} to be a collection in $\Col(\cC,\cV)$ equipped with a monoid structure \wrt plethysm.

\begin{example}
Let $\cC$ be the category of (non-empty) sets as in Example \ref{ex:symm col}.
Then $\cC_0=\bS$ is the groupoid of finite sets and a collection $X\in\Col(\cC,\cV)$ can be identified with a collection of objects $(X(n)\in\cV)_{n\ge1}$ such that $X(n)$ is equipped with an action of the symmetric group $\fS_n$.
A map $\vi:\bn=\set{1,\dots,n}\to\bmm=\set{1,\dots,m}$ in the above colimit should be surjective (we set $Y(\es)=0$).
A map $\psi:\bmm\to\bmm'$ from $\cC_0$ is a bijection and we can identify it with an element $\psi\in\fS_m$.
We obtain
\begin{equation}
(X\circ Y)(n)=
\bop_{m\ge1}
X(m)\ts_{\fS_m}
\rbr{
\bop_{\vi:\bn\to\bmm}Y(\vi)}
\end{equation}
which is the usual definition of plethysm between two symmetric collections \cite{getzler_operads}.
\end{example}


\begin{thebibliography}{10}

\bibitem{alexandrov_s}
Sergei Alexandrov, Jan Manschot, and Boris Pioline, \emph{S-duality and refined
  {BPS} indices}, Comm. Math. Phys. \textbf{380} (2020), no.~2, 755--810,
  \href{http://arxiv.org/abs/1910.03098}{{\ttfamily arXiv:1910.03098}}.

\bibitem{alexandrov_attractor}
Sergei Alexandrov and Boris Pioline, \emph{Attractor flow trees, {BPS} indices
  and quivers}, Advances in Theoretical and Mathematical Physics \textbf{23}
  (2019), no.~3, 627--699, \href{http://arxiv.org/abs/1804.06928}{{\ttfamily
  arXiv:1804.06928}}.

\bibitem{batanin_operadic}
Michael Batanin and Martin Markl, \emph{Operadic categories and duoidal
  {D}eligne's conjecture}, Adv. Math. \textbf{285} (2015), 1630--1687,
  \href{http://arxiv.org/abs/1404.3886}{{\ttfamily arXiv:1404.3886}}.

\bibitem{beaujard_vafa}
Guillaume Beaujard, Jan Manschot, and Boris Pioline, \emph{{V}afa-{W}itten
  invariants from exceptional collections}, 2020,
  \href{http://arxiv.org/abs/2004.14466}{{\ttfamily arXiv:2004.14466}}.

\bibitem{bousseau_flow}
Pierrick Bousseau and Hulya Arguz, In preparation.

\bibitem{bridgeland_introduction}
Tom Bridgeland, \emph{{A}n introduction to motivic {H}all algebras}, Adv. Math.
  \textbf{229} (2012), no.~1, 102--138,
  \href{http://arxiv.org/abs/1002.4372}{{\ttfamily arXiv:1002.4372}}.

\bibitem{bridgeland_scattering}
Tom Bridgeland, \emph{Scattering diagrams, {H}all algebras and stability
  conditions}, Algebraic Geometry (2017), 523--561,
  \href{http://arxiv.org/abs/1603.00416}{{\ttfamily arXiv:1603.00416}}.

\bibitem{cheng_dying}
Miranda~C.N Cheng and Erik~P Verlinde, \emph{Dying dyons don{\textquotesingle}t
  count}, Journal of High Energy Physics \textbf{2007} (2007), no.~09,
  070--070, \href{http://arxiv.org/abs/0706.2363}{{\ttfamily arXiv:0706.2363}}.

\bibitem{denef_supergravity}
Frederik Denef, \emph{Supergravity flows and {D}-brane stability}, Journal of
  High Energy Physics \textbf{2000} (2000), no.~08, 050--050,
  \href{http://arxiv.org/abs/hep-th/0005049}{{\ttfamily arXiv:hep-th/0005049}}.

\bibitem{denef_splita}
Frederik Denef, Brian Greene, and Mark Raugas, \emph{Split attractor flows and
  the spectrum of {BPS} {D}-branes on the quintic}, Journal of High Energy
  Physics \textbf{2001} (2001), no.~05,
  \href{http://arxiv.org/abs/hep-th/0101135}{{\ttfamily arXiv:hep-th/0101135}}.

\bibitem{denef_split}
Frederik Denef and Gregory~W. Moore, \emph{Split states, entropy enigmas, holes
  and halos}, Journal of High Energy Physics \textbf{2011} (2011), no.~11,
  \href{http://arxiv.org/abs/hep-th/0702146}{{\ttfamily arXiv:hep-th/0702146}}.

\bibitem{fresse_homotopy}
Benoit Fresse, \emph{{H}omotopy of operads and {G}rothendieck-{T}eichm\"{u}ller
  groups. {P}art 1}, Mathematical Surveys and Monographs, vol. 217, American
  Mathematical Society, Providence, RI, 2017, The algebraic theory and its
  topological background.

\bibitem{getzler_operads}
Ezra Getzler and J.~D.~S. Jones, \emph{{O}perads, homotopy algebra, and
  iterated integrals for double loop spaces}, 1994,
  \href{http://arxiv.org/abs/hep-th/9403055}{{\ttfamily arXiv:hep-th/9403055}}.

\bibitem{gross_canonicala}
Mark Gross, Paul Hacking, Sean Keel, and Maxim Kontsevich, \emph{Canonical
  bases for cluster algebras}, Journal of the American Mathematical Society
  \textbf{31} (2017), no.~2, 497--608,
  \href{http://arxiv.org/abs/1411.1394}{{\ttfamily arXiv:1411.1394}}.

\bibitem{joyal_une}
Andr\'e Joyal, \emph{{U}ne th\'eorie combinatoire des s\'eries formelles}, Adv.
  in Math. \textbf{42} (1981), no.~1, 1--82.

\bibitem{joyce_configurationsa}
Dominic Joyce, \emph{{C}onfigurations in abelian categories. {II}.
  {R}ingel-{H}all algebras}, Adv. Math. \textbf{210} (2007), no.~2, 635--706,
  \href{http://arxiv.org/abs/math.AG/0503029}{{\ttfamily
  arXiv:math.AG/0503029}}.

\bibitem{joyce_configurations}
\bysame, \emph{Configurations in abelian categories. {IV}. {I}nvariants and
  changing stability conditions}, Advances in Mathematics \textbf{217} (2008),
  no.~1, 125--204, \href{http://arxiv.org/abs/math/0410268}{{\ttfamily
  arXiv:math/0410268}}.

\bibitem{joyce_theory}
Dominic Joyce and Yinan Song, \emph{{A} theory of generalized
  {D}onaldson-{T}homas invariants}, Mem. Amer. Math. Soc. \textbf{217} (2012),
  no.~1020, iv+199, \href{http://arxiv.org/abs/0810.5645}{{\ttfamily
  arXiv:0810.5645}}.

\bibitem{kontsevich_stability}
Maxim Kontsevich and Yan Soibelman, \emph{{S}tability structures, motivic
  {D}onaldson-{T}homas invariants and cluster transformations}, 2008,
  \href{http://arxiv.org/abs/0811.2435}{{\ttfamily arXiv:0811.2435}}.

\bibitem{kontsevich_wall}
\bysame, \emph{{W}all-crossing structures in {D}onaldson-{T}homas invariants,
  integrable systems and {M}irror {S}ymmetry}, Lecture Notes of the Unione
  Matematica Italiana, Springer International Publishing, 2014,
  \href{http://arxiv.org/abs/1303.3253}{{\ttfamily arXiv:1303.3253}},
  pp.~197--308.

\bibitem{lack_operadic}
Stephen Lack, \emph{Operadic categories and their skew monoidal categories of
  collections}, High. Struct. \textbf{2} (2018), no.~1, 1--29,
  \href{http://arxiv.org/abs/1610.06282}{{\ttfamily arXiv:1610.06282}}.

\bibitem{loday_algebraic}
Jean-Louis Loday and Bruno Vallette, \emph{{A}lgebraic operads}, Grundlehren
  der Mathematischen Wissenschaften [Fundamental Principles of Mathematical
  Sciences], vol. 346, Springer, Heidelberg, 2012.

\bibitem{maclane_categories}
Saunders Mac~Lane, \emph{{C}ategories for the working mathematician}, second
  ed., Graduate Texts in Mathematics, vol.~5, Springer-Verlag, New York, 1998.

\bibitem{manschot_walla}
Jan Manschot, \emph{{W}all-crossing of {D}4-branes using flow trees}, Advances
  in Theoretical and Mathematical Physics \textbf{15} (2011), no.~1, 1--42,
  \href{http://arxiv.org/abs/1003.1570}{{\ttfamily arXiv:1003.1570}}.

\bibitem{mozgovoy_attractor}
Sergey Mozgovoy and Boris Pioline, \emph{Attractor invariants, brane tilings
  and crystals}, \href{http://arxiv.org/abs/2012.14358}{{\ttfamily
  arXiv:2012.14358}}.

\bibitem{reineke_harder-narasimhan}
Markus Reineke, \emph{{T}he {H}arder-{N}arasimhan system in quantum groups and
  cohomology of quiver moduli}, Invent. Math. \textbf{152} (2003), no.~2,
  349--368, \href{http://arxiv.org/abs/math/0204059}{{\ttfamily
  arXiv:math/0204059}}.

\end{thebibliography}
\providecommand{\bysame}{\leavevmode\hbox to3em{\hrulefill}\thinspace}
\providecommand{\href}[2]{#2}

\end{document}